\documentclass[english, 12pt]{article}
\usepackage[utf8]{inputenc}
\usepackage[a4paper]{geometry}
\usepackage[utf8]{inputenc}
\usepackage{amsmath}
\usepackage{amsfonts}
\usepackage[english]{babel}
\usepackage{amsthm}
\usepackage[ruled,vlined]{algorithm2e}
\usepackage{enumitem}
\usepackage{xcolor}
\usepackage{makecell,booktabs}
\usepackage{multirow}
\usepackage[title]{appendix}
\usepackage{accents}
\usepackage{csquotes}
\usepackage{float}
\usepackage[unicode=true,pdfusetitle,bookmarks=true,bookmarksnumbered=false,
bookmarksopen=false,breaklinks=true,pdfborder={0 0 1},backref=false,colorlinks=true]{hyperref}
\usepackage{graphicx}
\usepackage{epstopdf}
\usepackage{comment}
\usepackage{xfrac}
\usepackage{setspace} 
\DeclareGraphicsExtensions{.pdf}
\usepackage{caption}
\usepackage{stfloats}
\fnbelowfloat
\allowdisplaybreaks

\newtheorem{theorem}{Theorem}
\newtheorem{assumption}{Assumption}
\newtheorem{lemma}{Lemma}
\newtheorem{proposition}{Proposition}
\newtheorem{corollary}{Corollary}

\theoremstyle{definition}
\newtheorem{definition}{Definition}

\theoremstyle{remark}
\newtheorem{remark}{Remark}

\newtheorem*{L:PDA}{Lemma ~\ref{lemma:PDA}}
\newtheorem*{L:h_decreasing}{Lemma ~\ref{lemma:h_decreasing}}
\newtheorem*{L:m_f_existence}{Lemma ~\ref{lemma:m_f_existence}}
\newtheorem*{L:decreasing_Sequence}{Lemma ~\ref{lemma:decreasing_Sequence}}
\newtheorem*{Ex:ex_err_1}{Example ~\ref{ex_err_1}}
\newtheorem*{Ex:ex_err_2}{Example ~\ref{ex_err_2}}
\newtheorem*{Ex:ex_err_3}{Example ~\ref{ex_err_3}}
\newtheorem*{Ex:ex_err_4}{Example ~\ref{ex_err_4}}

\newtheorem*{cor:changing_epsilon_convergence}{Corollary ~\ref{cor:changing_epsilon_convergence}}
\numberwithin{equation}{section}

\newcommand{\ie}{\textit{i.e.,}}
\newcommand{\Real}{\mathbb{R}}
\newcommand{\norm}[1]{\left\|{#1}\right\|}
\DeclareMathOperator*{\dom}{dom}
\DeclareMathOperator*{\argmin}{\arg\!\min}
\DeclareMathOperator*{\argmax}{\arg\!\max}
\DeclareMathOperator{\prox}{prox}
\DeclareMathOperator{\st}{s.t}
\DeclareMathOperator{\Lev}{Lev}
\DeclareMathOperator{\dist}{dist}

\DeclareMathOperator{\Or}{or}
\DeclareMathOperator{\epi}{epi}

\DeclareMathOperator{\proj}{P}

\newcommand{\ba}{\mathbf{a}}
\newcommand{\bb}{\mathbf{b}}
\newcommand{\bp}{\mathbf{p}}
\newcommand{\bu}{\mathbf{u}}
\newcommand{\bv}{\mathbf{v}}
\newcommand{\bx}{\mathbf{x}}
\newcommand{\by}{\mathbf{y}}
\newcommand{\bz}{\mathbf{z}}
\newcommand{\bA}{\mathbf{A}}
\newcommand{\bQ}{\mathbf{Q}}

\newcommand{\bzero}{\mathbf{0}}

\newcommand{\bnu}{\boldsymbol {\nu}}

\newcommand{\ubar}[1]{\underaccent{\bar}{#1}}

\usepackage{cleveref}
\usepackage{hyperref}
\crefformat{footnote}{#2\footnotemark[#1]#3}

\title{Methodology and first-order algorithms for solving nonsmooth and non-strongly convex bilevel optimization problems}
\usepackage{authblk}
\author{Lior Doron}
\author{Shimrit Shtern\thanks{\href{mailto:shimrits@technion.ac.il}{shimrits@technion.ac.il}}}
\affil{Faculty of Data and Decision Sciences, Technion - Israel Institute of Technology, Haifa, Israel}
\date{}

\begin{document}

\maketitle
\abstract{Simple bilevel problems are optimization problems in which we want to find an optimal solution to an inner problem that minimizes an outer objective function. Such problems appear in many machine learning and signal processing applications as a way to eliminate undesirable solutions. 
In our work, we suggest a new approach that is designed for bilevel problems with simple outer functions, such as the $l_1$ norm, which are not required to be either smooth or strongly convex. In our new ITerative Approximation and Level-set EXpansion (ITALEX) approach, we alternate between expanding the level-set of the outer function and approximately optimizing the inner problem over this level-set. We show that optimizing the inner function through first-order methods such as proximal gradient and generalized conditional gradient results in a feasibility convergence rate of $O(1/k)$, which up to now was a rate only achieved by bilevel algorithms for smooth and strongly convex outer functions. 
Moreover, we prove an $O(1/\sqrt{k})$ rate of convergence for the outer function, contrary to existing methods, which only provide asymptotic guarantees.
We demonstrate this performance through numerical experiments. }
\normalsize
\section{Introduction}
We are interested in the following convex bilevel optimization problem (where we use the terminology of \textit{outer} and \textit{inner} levels). The \textit{outer} level is given by  the following constrained minimization problem:
\begin{equation}\tag{BLP}\label{prob:MNP}
    \min_{\bx\in X^*}\omega(\bx),
\end{equation}
where $\omega$ is convex and bounded from below on $X^*$. The set $X^*$ 
is the non-empty set of minimizers of the \textit{inner} level problem, which has the form of a classical convex composite model, given by:
\begin{equation}\tag{P}\label{prob:P}
    \min_{\bx\in\Real^n}\{\varphi(\bx):=f(\bx)+g(\bx)\},
\end{equation} where $f$ is convex and continuously differentiable with a Lipschitz-continuous gradient, and $g$ is an extended real-valued, closed, and convex function. We denote the optimal value of \eqref{prob:P} as $\varphi^*$ and the optimal value of \eqref{prob:MNP} as $\omega^*$. The exact assumptions on the structure of this problem will be given in Section~\ref{section:the_method}.

This type of problem, also known sometimes as \textit{simple bilevel programming} or \textit{hierarchical convex optimization}, is a particular case of a general bilevel programming problem (as described in \cite{SBP}). It is considered in many machine learning, signal processing, and regression applications as a key approach to solve underdetermined problems. Indeed, in order to find a sparse solution to \eqref{prob:P} we would be interested in solving the bilevel problem \eqref{prob:MNP} with $\omega(\bx)=\|\bx\|_1$  \cite{Basis_Pursuit} and in order to find a dense solution we would choose $\omega(\bx)=\|\bx\|_2^2$  \cite{Tikhonov}. However, since these bilevel problems are generally hard to solve, a common approach is to incorporate $\omega$ as a regularization term and solve the relaxed problem $\min_{\bx}\{ \varphi(\bx)+\lambda\omega(\bx)\}$. Such a regularization approach leads to the famous LASSO model when $\omega$ is the $l_1$ norm \cite{Basis_Pursuit,LASSO} and to the ridge regression model when $\omega$ is the $l_2$ norm \cite{Ridge,Tikhonov}. While in some cases there exists a small enough value $\lambda$ for which the regularized and bilevel problems are equivalent \cite{Tikhonov}, finding this $\lambda$ is not a trivial task.

Motivated by the case $\omega(\bx)=\|\bx\|_1$, in this work we will focus mainly on the cases where $\omega$ is convex but not necessarily smooth or strongly convex.

\subsection{Main challenges of Bilevel optimization problem}\label{subsection:challenges}
The \eqref{prob:MNP} can be reformulated as the following simple convex optimization problem: \begin{equation}\label{prob:MNP'}\tag{BLP'}\begin{aligned}
    &\min_{\bx\in\Real^n} && \omega(\bx) \\
    &\st &&\varphi(\bx)\leq \varphi^*,
\end{aligned}\end{equation}
where $\varphi^*$ is the exact (or approximated) optimal value of \eqref{prob:P}. However, while this problem is convex, problem \eqref{prob:MNP'} inherently does not satisfy the necessary regularity conditions (\textit{e.g} Slater's condition) to apply strong duality or KKT conditions \cite[Theorem 3.1.17]{Convex_Opt_Book}. Consequently, off-the-shelf algorithms to solve this problem that rely on primal-dual relations, such as interior-point algorithm, may fail.
Moreover, even when $\omega$ is smooth, classical first-order algorithms, such as projected gradient, cannot be applied because the orthogonal projection over $X^*=\Lev_{\varphi}(\varphi^*)$ may be hard to compute. These difficulties persist even when $\varphi^*$ is approximated to high accuracy, and numerical difficulties prevent solving it using existing algorithms.

\subsection{Existing methods for Bilevel convex optimization}
During the last two decades, several first-order methods have been suggested to solve \eqref{prob:MNP}. We highlight some of them here.
An extensive review of algorithms for simple bilevel optimization is available in \cite{SBP}.

One class of algorithms for solving problem \eqref{prob:MNP} is based on applying a first-order operator (or a sequence of them) on the regularized problem  
\begin{equation*}\phi_k(\bx)=\varphi(\bx)+\lambda_k \omega(\bx),\end{equation*}
at each iteration $k$, where the regularization parameter sequence $\{\lambda_k\}_{k\in\mathbb{N}}$ satisfies $\lim_{k\rightarrow \infty}\lambda_k=0$ and $\sum_{k=1}^\infty \lambda_k=\infty$. The algorithms in this class differ from each other in their assumptions and the operator used.
In \cite{Solodov}, Solodov suggests the Iterative Regularized Projected Gradient (IR-PG) method. IR-PG applies a projected gradient step
to $\phi_k$ at iteration $k$, subject to the additional assumptions that  $\omega$ is continuously differentiable with an $L_\omega$-Lipschitz continuous gradient and $g(\bx)=\delta_C(\bx)$, where $\delta_C$ is the indicator function of a closed convex set $C$.\footnote{\label{note1}Both IR-PG and MNG, as well as their convergence results, can easily be extended to the case where $g$ is a proximal friendly function \cite[chapter 6]{FO_Book}.} In \cite{helou2017},  Helou and Sim\~{o}es suggest to apply a three step variation of the $\epsilon$-subgradient method, under the additional assumptions that $f$ and $\omega$ have bounded (sub)gradients, and $g=\delta_C(\bx)$ as before. In this algorithm, the first step is an accelerated-gradient step on $f$, the second step is a subgradient step or proximal step on $\omega$, and the third step is projection onto set $C$. 
These two works are first-order variants of an algorithm first suggested by Cabot \cite{Cabot2005} and recently extended by Dutta and Pandit \cite{SBP}, in which a proximal point step is applied to $\phi_k$ at each iteration.  Although the proximal point algorithm is very general, as it does not assume any structure on either $\varphi$ or $\omega$, its iterations may be extremely computationally expensive, as the proximal operator of a sum of convex functions is generally not simple to calculate even when computing the proximal operator for each function individually is easy \cite{prox_sum}. Moreover, all the methods mentioned above prove asymptotic convergence to a solution of the bilevel problem, however they do not provide rates for this convergence.
In \cite{MNG}, Beck and Sabach present the Minimal Norm Gradient (MNG) method for the case where $\omega$ is a $\sigma_\omega$-strongly convex and continuously differentiable function with an $L_\omega$-Lipschitz continuous gradient, for example $\omega(\bx)=\|\bx\|_2^2$. This is the first method to provide any rate of convergence result for problem~\ref{prob:MNP}. MNG uses a notion of cutting planes, where at each iteration two half-spaces are constructed, and the \textit{outer} function is  minimized on the intersection of these two half-spaces. They show that MNG converges to the optimal solution of the bilevel problem and has an $O(1/\sqrt{k})$ convergence rate for the \textit{inner} problem, that is, for obtaining a feasible solution for the bilevel problem. It is also important to notice that even though constructing the half-spaces involves standard first-order computations, the method requires solving an optimization problem in each iteration, which in some cases cannot be done analytically.\cref{note1}

In \cite{Big_SAM}, under similar assumptions to \cite{MNG}, Sabach and Shtern suggested the Bilevel Gradient Sequential Averaging Method (BiG-SAM). The algorithm takes a convex combination between a proximal gradient step over the \textit{inner} function and a gradient step over the \textit{outer} function, \textit{i.e,} the general step of the method is \begin{equation*}
    \bx^{k+1}=\alpha_k(\bx^k-s\nabla \omega(\bx^k))+(1-\alpha_k)\prox_{tg}(\bx^k-t\nabla f(\bx^k)),
\end{equation*}
where the sequence $\{\alpha_k\}$ satisfies the following properties:
$\lim_{k\rightarrow\infty}\alpha_k=0$, $\sum_{k=1}^\infty\alpha_k=\infty$, $\lim_{k\rightarrow\infty}\frac{\alpha_k}{\alpha_{k+1}}=1.$ 
Their algorithm uses the non-expansiveness of the proximal-gradient operator over functions with  Lipschitz-continuous gradient, and the contraction property of the gradient step over strongly-convex functions with Lipschitz-continuous gradient to get a $O(1/k)$ convergence rate for the \textit{inner} problem.
The authors also show that through smoothing, their algorithm can be adapted to the case where $\omega$ is the sum of a Lipschitz-continuous but nonsmooth function and a smooth strongly convex function, \textit{e.g.}, the elastic net function $\omega(\bx)=\|\bx\|_1+\rho\|\bx\|_2^2$. They show that, in this case, the algorithm converges to the optimal solution of \eqref{prob:MNP} where $\omega$ is replaced by its Moreau envelope \cite[chapter 6.7]{FO_Book}, and the  iteration complexity for obtaining an $\varepsilon$-feasible solution for the bilevel problem is  $O(\frac{1}{\varepsilon\delta^2})$, where $\delta>0$ is the required uniform accuracy on the approximation of the \textit{outer} objective function.
In \cite{iBiG}, Shehu et al. present iBiG-SAM, a variation of BiG-SAM with an inertial step. While the convergence of iBiG-SAM was proven without a convergence rate, \cite{iBiG} provides several numerical examples indicating it may outperform BiG-SAM.

In \cite{Incremental}, Amini and Yousefian extended the IR-PG method for the case where $\omega$ is strongly convex but not necessarily differentiable (for example the 'elastic net' function). They also deviate from the classical composite model for $\varphi$, and assume instead that $f(\bx)=\sum_{i=1}^m f_i(\bx)$, where $f_i$ are proper, closed, real-valued, and convex for all $i\in\{1,\ldots,m\}$, and $g=\delta_C$ where $C$ is compact. Each iteration of their algorithm, named Iterative Regularized Incremental projected subGradinet (IR-IG) method, is an incremental projected subgradient step \cite[section 8.4]{FO_Book} over the sum of functions $\sum_{i=1}^m(f_i(\bx)+\frac{\lambda_k}{m}\omega(\bx))$. The authors showed that by choosing stepsizes and $\lambda_k$ properly, the algorithm converges to the solution of the problem, and the inner problem value converges at a rate of $O(1/k^{0.5-\beta})$ for any fixed $\beta\in(0,0.5)$.

As seen above, existing algorithms require $\omega$ to be strongly convex in order to obtain rate guarantees. Moreover, all existing rates are in terms of inner-function objective, that is, in terms of feasibility rather than optimality of problem  \eqref{prob:MNP}.


\subsection{Contribution}
 In this paper, we present a new optimization scheme called \textit{ITerative Approximation and Level-set EXpansion} (ITALEX) to solve the bilevel problem \eqref{prob:MNP}. The method alternates between applying an approximate optimization oracle to solve the problem of minimizing a surrogate inner function $\hat{\varphi}^\alpha$ over a given level-set of the outer function $\omega$, and enlarging the aforementioned level-set. By carefully enlarging the level-set we guarantee the method will provide an $\varepsilon$-feasible and $\sqrt{\varepsilon}$-optimal solution for the bilevel problem, that is, for any $\varepsilon>0$ the method outputs $\tilde{\bx}$ such that
 \begin{equation}\label{eq:result}
 \varphi(\tilde{\bx})\leq \varphi^*+\varepsilon,\,  \omega(\tilde{\bx})- \omega^*\leq O(\sqrt{\varepsilon}).\end{equation}
 To summarize the novelty and contribution of the paper:
\begin{itemize}
    \item We propose a new scheme to solve bilevel problems, called ITALEX, that assumes neither strong convexity nor differentiability of $\omega$, and can be easily applied to important bilevel problems, such as ones where $\omega$ is an $l_p$ norm.
    \item ITALEX is the first scheme to provide a rate of convergence for the outer function's optimality gap. 
    \item We prove that under our mild assumptions, using generalized conditional gradient or proximal gradient as our approximate optimization oracle within ITALEX results in an iteration complexity of $O(1/\varepsilon)$ to obtain an $\varepsilon$-feasible and $\sqrt{\varepsilon}$-optimal solution. 
    {In particular, our assumptions on the outer function $\omega$ generalize the setting in which it is strongly convex.}
    Thus, ITALEX achieves the best known feasibility convergence rate, previously obtained by BiG-SAM, which utilizes the stronger assumption that $\omega$ is both smooth and strongly convex.  
    \item We show that a variation of our scheme can be used when $\varphi$ is smooth in order to also guarantee super-optimality with respect to the \textit{outer} problem, \textit{i.e,} 
    $\omega(\tilde{\bx})\leq\omega^*$. 
    \item We demonstrate the applicability of this approach in numerical experiments, where we compare ITALEX to existing methods.
\end{itemize}

A comparison between the assumptions and convergence rates of ITALEX and other first-order methods that provide rate of convergence results is presented in Table~\ref{table:Comparison}.

\begin{table}[ht]
\begin{minipage}{\textwidth}
\centering
\small
\begin{tabular}{p{1.7cm}| p{3.5cm} p{3.5cm} p{2cm} p{2cm}}
 \hline
 Method & \raggedright{$\varphi=f+g$\\ properties} &$\omega$ properties & Convergence rate for \eqref{prob:P} & Convergence to $\omega^*$ \\ [0.5ex] 
 \hline\hline
 MNG \cite{MNG} & Classical composite & \raggedright{Smooth, strongly convex} & $O\left(\frac{1}{\sqrt{k}}\right)$ & Asymptotic\\ \hline
 BiG-SAM \cite{Big_SAM} & Classical composite & \raggedright{Smooth, strongly convex} & $O\left(\frac{1}{{k}}\right)$ & Asymptotic\\ \hline
 IR-IG \cite{Incremental} & \raggedright{$f$ is a finite sum, $g=\delta_C$, $C$ compact} & \raggedright{Strongly convex} & $O\left(\frac{1}{k^{0.5-\beta}}\right)$  $\beta\in(0,0.5)$ & Asymptotic \\\hline
\multirow{3}{1.5cm}{{\textbf{ITALEX}} [This paper]} &  \raggedright{Classical composite}& \multirow{3}{3.5cm}{\raggedright{Norm-like function (Assumption \ref{assumption:omega})}} 
& \multirow{3}{2cm}{{$O\Bigl(\frac{1}{k}\Bigr)$ }}& $O\left(\frac{1}{\sqrt{k}}\right)$\\%
\cline{2-2}\cline{5-5}
&Smooth&  & &{\raggedright{Super-optimal}}\\
\hline
\end{tabular}
\end{minipage}
\caption{Comparison between bilevel optimization methods}
\label{table:Comparison}
\end{table}

\subsection{Mathematical notations}
In this paper we use the following standard notation: 
Vectors and matrices are written in bold. Given closed and convex set $C$, we denote its indicator function by $\delta_C$, the orthogonal projection of $\bx$ onto $C$ is denoted by $\proj_C(\bx)=\argmin\{\|\by-\bx\|^2:\by\in C\}$ and the distance between $\bx$ and $C$ is denoted as $\dist(\bx,C)$. Moreover, if $C$ is compact we can define its diameter $\mathcal{D}_C={\max}\{\norm{\bu-\bv}:\bu,\bv\in C\} $. Given function $\omega$, we denote its $\alpha$-Level set as $\Lev_{\omega}(\alpha)=\{\bx:\omega(\bx)\leq\alpha\},$ its domain as $\dom(\omega)$ and its epigraph as $\epi(\omega)$. The subdifferential set at point $\bx$ is denoted as $\partial \omega(\bx)$.
Unless stated otherwise $\|\cdot\|$  and $\langle \cdot,\cdot\rangle$ denote the euclidean norm and inner product, respectively. Given a positive definite matrix $\bQ$ we denote $\|\bx\|_\bQ\equiv \sqrt{\bx^{\top}\mathbf{ Q x}}$. 
{For any vector $\bx\in \Real^n$, and scalar $r>0$, we denote the closed ball of radius $r$ centered in $\bx$  by $\mathcal{B}(\bx,r)=\{\by\in \Real^n:\norm{\by-\bx}\leq r\}$.}
For a matrix $\mathbf{D}$ we denote its minimal and maximal eigenvalue as $\lambda_{\min}(\mathbf{D})$ and $\lambda_{\max}(\mathbf{D})$, respectively. The set $\{1,2,...,m\}$ {for some integer $m$} is denoted as $[m]$, and for a real number $r$, $[r]_+=\max\{r,0\}$.
\section{ Preliminaries}\label{section:preliminaries}

\subsection{Generalized Conditional and Proximal Gradient Methods and the Surrogate Optimality Gap}
In our analysis we consider a continuous optimality measure related to some optimization algorithm. One possible continuous optimality measure, associated with the generalized conditional gradient algorithm, is the surrogate optimality gap. 
\begin{definition}\label{S_def}  \cite[Definition 13.2]{FO_Book}: The surrogate optimality gap for a function $\varphi\equiv f+g$, where $f$ is convex, continuously differentiable with $L_f$-Lipschitz continuous gradient, and $g$ is an extended real-valued convex function with compact domain, is given by
\begin{equation}\label{eq:S_def}
S(\bx)=\max_{\bp}\Bigl\{\Bigl\langle\nabla f(\bx),\bx-
\bp\Bigr\rangle+g(\bx)-g(\bp)\Bigr\}
\end{equation}
\end{definition}
The surrogate optimality gap has a dual role in our analysis as both an optimality measure and as a sufficient decrease measure. Indeed, the surrogate optimality gap $S(\bx)$ provides an upper bound on the true optimality gap, as shown in the following lemma,
\begin{lemma}\label{S_lemmas}{\bf{\cite[Lemma 13.12, Theorem 13.6]{FO_Book}}} For any $\bx\in\dom(g)$
\begin{equation*} S(\bx)\geq \varphi(\bx)-\min_{{\by}\in\Real^n}\{\varphi(\by)\},\end{equation*} moreover, $S(\bx)=0$ if and only if $\bx$ is {a minimizer} of $\varphi$.
\end{lemma}
To view the role of $S(\bx)$ as a sufficient decrease measure we turn to the \textbf{Generalized Conditional Gradient} (GCG) method, in which $S(\bx)$ arises naturally as part of the convergence analysis. The GCG method \cite{GCG} was developed for optimizing composite functions, such that $\dom(g)$ is compact. The GCG updating step is given by
\begin{equation*}
    \bx^{k+1}=\bx^k+\eta_k\left(\bp(\bx^k)-\bx^k\right),\end{equation*}
    where \begin{equation*} \bp(\bx)\in\argmin\left\{\langle\nabla f(\bx),\bp\rangle+g(\bp)\right\}.
\end{equation*}
Note that $\bp(\bx)$ is a maximizer of the optimization problem defined in \eqref{eq:S_def}. The following result shows that for a specific step-size choice, $S(\bx)$ serves as a sufficient decrease measure for the GCG step.  

\begin{lemma} \label{S_lemmas_sufficient_decrease}{\bf{\cite[Lemma 13.8]{FO_Book}}}: Let $\{\bx^k\}_{k\in\mathbb{N}}$ be the sequence generated by the GCG method with adaptive stepsize $\eta_k=\min\left\{1,\frac{S(\bx^k)}{L_f\|\mathbf{p}(\bx^k)-\bx^k\|^2}\right\}$ or with stepsize chosen by exact line search. Then, \begin{equation*}
\varphi(\mathbf{\bx}^k)-\varphi(\mathbf{x}^{k+1})\geq\frac{1}{2}\min\left\{S(\bx^k),\frac{S^2(\bx^k)}{L_f \mathcal{D}_{\dom(g)}^2}\right\}.\end{equation*}
\end{lemma}
Thus, the connection between the surrogate optimality gap $S(\bx)$ and the GCG method, naturally leads to the use of the GCG method as a descent scheme in our method.

Another commonly used descent scheme that we are going to employ as part of our method is the \textbf{Proximal Gradient} (PG) method. The PG method can also be applied on the classical composite model discussed above. The PG method consists of the following update step \begin{equation}\label{eq:Proximal_Gradien}
    \bx^{k+1}=T(\bx^k)\equiv \prox_{\frac{1}{L_k}g}\left(\bx^k-\frac{1}{L_k}\nabla f(\bx^k)\right),
\end{equation}
where \begin{equation}
    \prox_g(\bx)=\argmin_{\mathbf{u}\in\Real^n}\left\{g(\mathbf{u})+\frac{1}{2}\|\bx-\mathbf{u}\|^2\right\},
\end{equation}
and $L_k$ is either constant in $[L_f,\infty)$ or chosen by a backtracking procedure.
The surrogate optimality gap $S(\bx)$ can also be connected to the proximal gradient method updating step by \cite[Section 2.2.2]{PDA}.
Since $S(\bx)$ is not computed as part of the PG method, we can establish another optimality measure $\tilde{S}(\bx)$, which only depends on $T(\bx)$. Utilizing Lemma~\ref{S_lemmas} and \cite[Lemma 2.6]{PDA}, we show that, similarly to the relation between $S(\bx)$ and the GCG method, $\tilde{S}(\bx)$ also serves as a sufficient decrease measure for the PG method. {
\begin{lemma}\label{corollary:proximal_bound}
        Let $\{\bx^k\}_{k\in\mathbb{N}}$ be the sequence generated by the PG method as in \eqref{eq:Proximal_Gradien}. Assume that 
        $D(\bx^k)\equiv\mathcal{D}_{\Lev_{\varphi}(\varphi(\bx^k))}\leq\bar{D}<\infty$ for all $k\in\mathbb{N}$, and denote $$\tilde{S}(\bx^k)=2\max\left\{\varphi(\bx^k)-\varphi\left(\bx^{k+1}\right),\sqrt{\frac{L_k}{2}D(\bx^k)^2\left(\varphi(\bx^k)-\varphi\left(\bx^{k+1}\right)\right)}\right\}. $$
        {Let $\varphi^*=\min_{\by\in \Real^n}\varphi(\by)$.}
        Then, for all $ k\in \mathbb{N}$ \begin{equation*}
        \varphi(\bx^k)-{\varphi^*}
        \leq \tilde{S}(\bx^k).
        \end{equation*}
    \end{lemma}}
    \begin{proof}
        {Let $\bx^*$ be an optimal solution of $\varphi$. Thus,  $\bx^*,\bx^{k}\in \Lev_{\varphi}(\varphi(\bx^k))$, and $\norm{\bx^k-\bx^*}\leq D(\bx^k)$. 
        For any given point $\bx\in\dom(g)$, define $S_{D(\bx)}(\bx)=\min_{\bp\in\mathcal{B}(\bx,D(\bx))}\{\langle\nabla f(\bx),\bp\rangle+g(\bp)\}$, which can be viewed as the surrogate optimality gap of the function $\hat{\varphi}_{\bx}=\varphi+\delta_{\mathcal{B}(\bx, D(\bx))}$. Since $\bx^*\in \mathcal{B}(\bx, D(\bx))$, it is also  a minimizer of $\hat{\varphi}_{\bx}$, and by Lemma~\ref{S_lemmas} we have $S_{D(\bx)}(\bx)\geq \hat{\varphi}_{\bx}(\bx)-\hat{\varphi}_{\bx}(\bx^*)={\varphi}(\bx)-{\varphi}(\bx^*)$.
       Since PG is a descent method, $T(\bx)\in\Lev_{\varphi}(\varphi(\bx))$ \cite[Lemma 10.14]{FO_Book}, and so it is equal to the proximal gradient step applied to {$\hat{\varphi}_{\bx}$}, \ie\;} 
\begin{equation*}
   T(\bx)=\argmin_{\bu}\{f(\bx)+\langle \nabla f(\bx), \mathbf{u}-\bx\rangle +\frac{L_f}{2}\|\mathbf{u}-\bx\|^2+g(\mathbf{u})+\delta_{\mathcal{B}(\bx,D(\bx))}(\bu)\}.
\end{equation*}
 Thus, we can conclude directly from the definition of PDA step in \cite[Definition 2.1]{PDA} and the fact that the proximal gradient step is a $1$-PDA \cite[Section 2.2.2]{PDA}, that for any $t\in[0,1]$ 
 \begin{equation*} \varphi(T(\bx))\leq\varphi(\bx)-t\cdot S_{D(\bx)}(\bx)+ \frac{L_k D(\bx)^2}{2}t^2,\; \forall t\in[0,1],
            \end{equation*}
          Choosing $t=\min\left\{1,\sqrt{\frac{2(\varphi(\bx^k)-\varphi(\bx^{k+1}))}{L_f D(\bx^k)^2}}\right\}$
        derives the wanted result.
    \end{proof}
    Lemma~\ref{corollary:proximal_bound} is the proximal gradient analogue to Lemma~\ref{S_lemmas}. Its proof uses the decrease property of the PG method to show that the defined $\tilde{S}(\bx)$ is indeed an optimality measure.
\subsection{Error bounds} \label{subsection:errorbounds}
In our analysis, we utilize a global error-bound on the function $\omega$. For this we assume that $\omega$ is bounded from below by $\ubar{\omega}=\inf_{\bx\in\Real^n}(\omega(\bx))<\infty$.
The global error-bound bounds the distance between a point and a level set of a function, using the difference in the function's values.\begin{definition}\label{def:error_bound}
    For a convex function $\omega$ 
    we say it has a {$\kappa$-power} $\gamma$-global error-bound for some {$\kappa\in(0,2]$ and} $\gamma>0$ if for any $\bx\in\Real^n$ and any $\alpha\geq \ubar{\omega}$
    \begin{equation*}\label{eq:def_error_bound}
     {\dist(\bx,\Lev_{\omega}(\alpha)){^\kappa}\leq \gamma [\omega(\bx)-\alpha]_+.}
\end{equation*}
\end{definition}
{Note that if $\omega$ is real-valued $\sigma$-strongly convex function, it has such an error bound with $\kappa=2$.}
    {\begin{proposition}
    Let $\omega$ be a $\sigma$-strongly convex function, then $\omega$ has a global error-bound with $\kappa=2,\gamma=\frac{2}{\sigma}.$
    \end{proposition}
    \begin{proof}
        Set $\alpha\geq\ubar{\omega}$. If $\bx\in \Lev_\omega(\alpha)$ the statement trivially holds. Otherwise, let $\bx$ satisfy $\omega(\bx)>\alpha$. Denote $\bx^\alpha=\proj_{\Lev_{\omega}(\alpha)}(\bx)$ and let  $\bu=\argmax_{\bv\in\partial\omega(\bx^\alpha)}\langle\bv,\bx-\bx^{\alpha}\rangle$. By 
         \cite[Theorem 3.26]{FO_Book}, the directional derivitive in direction $\bx-\bx^\alpha$ satisfies $\omega'(\bx^\alpha,\bx-\bx^\alpha)=\langle\bu,\bx-\bx^\alpha\rangle$, and since $\omega(\tilde{\bx})>\alpha$ for all $\tilde{\bx}\in(\bx^\alpha,\bx]$,  $\omega'(\bx^\alpha,\bx-\bx^\alpha)\geq 0$. Therefore,  using the first order characterization of strong convexity \cite[Theorem 5.24]{FO_Book} we obtain \begin{align*}
        \omega(\bx)&\geq\omega(\bx^{\alpha})+\langle\bu,\bx-\bx^{\alpha}\rangle+\frac{\sigma}{2}\norm{\bx-\bx^\alpha}^2\geq\alpha+\frac{\sigma}{2}\dist^2\left(\bx,\Lev_{\omega}(\alpha)\right),
        \end{align*}
        which is equivalent to the desired result.
    \end{proof}}
{Moreover, Definition~\ref{def:error_bound} is more general than strong convexity, and non strongly convex functions may also have such a global error bound.} 
{Specifically, the following result, which is a direct consequence of \cite[Theorem 1]{Error_bounds} for function $\Gamma=\omega(\bx)-\alpha$, provides an easy method for verifying a $1$-power global error bound for a given function $\omega$, and finding the associated constant $\gamma$.}
    \begin{proposition}\label{proposition:error_bound_subgradient}
        Let $\omega$ be a closed, proper, and convex function, and let
    \begin{equation*}\label{error_bound_subgrad}
        \tau=\inf_{\bx,\mathbf{v}}\{\|\mathbf{v}\|:\mathbf{v} \in \partial \omega(\bx), \omega(\bx)>\ubar{\omega}\}
    \end{equation*}
    If $\tau>0$, then $\omega$ has a global error bound with {$\kappa=1$} and $\gamma=\tau^{-1}$, \ie 
    $$\dist(\bx, \Lev_{\omega}(\alpha))\leq \gamma [\omega(\bx)-\alpha]_+, \;\forall \bx\in \Real^n,\;\alpha\geq \ubar{\omega}.$$
    \end{proposition}
    
    
    In Table~\ref{tbl:error_bound}, we show several examples of functions $\omega$ that have a global error bound and their associated {$\kappa$ and }$\gamma$.
    As can be seen by these examples, the existence and the derivation of $\gamma$ can be computed easily in many important cases, including the interesting cases where $\omega$ is an $\ell_p$ norm.
    \begin{table}[ht]
    \begin{minipage}{\textwidth}
    \renewcommand{\arraystretch}{1.5}
    \centering
    \footnotesize
    \begin{tabular}{lllll}
    \hline
    $\omega$ & Parameters & $\Lev_{\omega}(\alpha)$ & {$\kappa$} & $\gamma$\\
    \hline
    {$\eta(\bx)+\frac{\sigma}{2}\norm{\bx}^2$} & {$\eta$ convex, $\sigma> 0$} & {Strongly-convex} & {2} & {$\frac{2}{\sigma}$}\\
    $\max\limits_{j\in[m]}\{\langle \ba_j,\bx\rangle- b_j\}$ & $\ba_j\in \Real^n,\; j\in [m],\;\bb\in \Real^m$ & Polytope & {1} & $\max_{j\in[m]}\{\norm{\ba_j}^{-1}\}$\\
    $\norm{\bx-\bx_0}_{\bQ}\vphantom{\max\limits_{j\in[m]}a}$ & $\bQ\in \mathbb{S}^n,\;\bQ\succ 0, \;\bx_0\in \Real^n$ & Ellipsoid & {1} & $\frac{1}{\sqrt{\lambda_{\min}(\bQ)}}$\\
    \multirow{2}{*}{$\norm{\bx-\bx_0}_p\phantom{\max\limits_{j\in[m]}}$} & \multirow{2}{*}{$\bx_0\in \Real^n,\; \begin{array}{c} 1\leq p\leq 2\\p>2\end{array}$}& \multirow{2}{*}{$\ell_p$-ball} & {\multirow{2}{*}{1}} & 
    \multirow{2}{*}{ $\begin{array}{l} 1\\n^{\frac{1}{2}-\frac{1}{p}} \end{array} $}\\
     &&&&\\
    \multirow{2}{*}{$\norm{x}_1+\rho\norm{x}_2^2 \vphantom{\max\limits_{j\in[m]}}$} & \multirow{2}{*}{$\rho>0$} & \multirow{2}{*}{Elastic net 'ball'} &{1}& $1$\\
    &&& {2} &{$\frac{1}{\rho}$}\\
    \hline
    \end{tabular}
    \end{minipage}
    \caption{{Error bound constants}}\label{tbl:error_bound}
    \end{table}
\section{The method and its implementation}\label{section:the_method}
\subsection{Approach and general scheme}\label{subsec:general_scheme}
Our approach is based on a formulation of problem \eqref{prob:MNP}, which is related to the formulation \eqref{prob:MNP'}. The main idea behind the method is that in many important bilevel problems, projecting onto the level sets of $\omega$ is much easier than projecting onto $X^*$, the implicitly defined feasible set. Take for example the $l_1$ norm, even though it is not differentiable and not strongly convex, its level sets are $l_1$ balls, which are amenable to linear oracles and projections. This insight encouraged us to explore a different formulation of \eqref{prob:MNP} for which we will later derive our algorithm. A key element of this formulation is the function  $H:\Real^{2n+1}\rightarrow(-\infty,+\infty]$, \begin{equation*}
    H(\bx,\bz,\alpha)=\varphi(\bx)+\norm{\bz-\bx}^2+\delta_{\epi(\omega)}(\bz,\alpha)\equiv\begin{cases} \varphi(\bx)+\norm{\bx-\bz}^2, & \omega(\bz)\leq \alpha\\ \infty, & \omega(\bz)>\alpha\end{cases}
\end{equation*}
Note that $H(\bx,\bz,\alpha)$ is an extended real-valued convex function. For our analysis we will also use the extended real-valued one dimensional function $h:\Real\rightarrow(-\infty,+\infty]$ given by
 
\begin{equation}\label{def:h}\tag{$\text{P}_\alpha$}h(\alpha)=\min_{\bx\in\Real^n,\bz\in\Real^n} H(\bx,\bz,\alpha)=\min_{\bx\in\Real^n}\left\{\varphi(\bx)+\dist(\bx,\Lev_\omega(\alpha))^2\right\}.
\end{equation}
For the following results and definitions we also use the notation $\ubar{\omega}=\inf_\bx\{\omega(\bx)\}$ (which may be equal to $-\infty$).
The following result presents some key properties of $h(\cdot)$, which we will use in our analysis, and are proven in Appendix~\ref{appendix:h_properties}
\begin{lemma}\label{lemma:h_decreasing}
Given the definition of $h(\cdot)$ in \eqref{def:h} the following holds:
\begin{enumerate}[label=(\roman*)]
\item $h(\cdot)$ is a one dimensional convex function
\item $h(\cdot)$ is nonincreasing, and is decreasing for any $\alpha\in(\ubar{\omega},\omega^*)$.
\end{enumerate}
\end{lemma}

As in the \eqref{prob:MNP'} formulation, we denote $\varphi^*$ to be the optimal value of \eqref{prob:P} and $X^*$ to be its optimal solution set. Since $\varphi(\bx)=\varphi^*$ for all $\bx\in X^*$, the goal of our algorithm is to find the solution to the following optimization problem, which has the same optimal value as the bilevel problem \eqref{prob:MNP}
\begin{equation}\tag{\text{BLP}$_\alpha$}\label{MNP_alpha}
    \min_{\bx,\bz,\alpha}\left\{\alpha: H(\bx,\bz,\alpha)\leq \varphi^*\right\}.
\end{equation}

The general scheme of the algorithm is based on two main operations:\begin{enumerate}[label=(\roman*)]
    \item  Approximately optimizing problem~\eqref{def:h} for a given $\alpha$ level set of the \textit{outer} function  $\omega$.
    \item Expanding the the level set by increasing $\alpha$ while maintaining $\alpha\leq \omega^*$, where $\omega^*$ is the optimal value of \eqref{prob:MNP}.
\end{enumerate}
We will first show the general scheme of our method, \textit{ITerative Approximation and Level-set EXpansion} (ITALEX), and then, under some mild assumptions, provide concrete implementations of the operators used in the algorithm. 
We now define two oracles that lie at the core of our algorithm.
\begin{definition}\label{Def:APP_Oracle} The operator  $\mathcal{O}^{\omega,\varphi}(\bx,{\bz},\alpha,\bar{\varphi},\varepsilon)$ is called an \emph{Approximation Oracle} if for any $\varepsilon>0,\bar{\varphi}\geq\varphi^*$, $\alpha\geq \ubar{\omega}$, { $\bx\in\dom(\varphi),\bz\in\Lev_{\omega}(\alpha)$} it returns a triple $(\rho,\mathbf{y}_1,\mathbf{y}_2)\in \Real_+\times \dom(g)\times \Lev_{\omega}(\alpha)$ {(the value of which may depend on $\bx$ and $\bz$)} such that
\begin{align*}
    \rho \in \left[\frac{\varepsilon}{2},h(\alpha)-\bar{\varphi}\right] \quad\Or\quad
    \varphi(\by_1)+\norm{\by_1-\by_2}^2\leq \bar{\varphi}+\varepsilon,\;\rho=0.
\end{align*}
\end{definition}
From the definition of the approximation oracle, it is clear that a strictly positive $\rho$ can be returned only if $h(\alpha)-\bar{\varphi}\geq \varepsilon/2$, and $\rho=0$ can only be returned if there exists a $\by_1\in \dom(g)$ and $\by_2\in \Lev_{\omega}(\alpha)$ such that $h(\alpha)\leq \varphi(\by_1)+\norm{\by_1-\by_2}^2\leq \bar{\varphi}+\varepsilon$. Thus, in the case where $\varepsilon/2 \leq h(\alpha)-\bar{\varphi}\leq \varepsilon$ both outputs are possible. {Note that the inclusion of $\bx$ and $\bz$ in the definition of the oracle allows us to use these inputs as a starting point for the algorithm to find $\by$ and $\rho$, a fact that would prove crucial for obtaining the convergence rates in Section~\ref{subection:Exp}.}
\begin{definition}\label{Def:EXP_Oracle} 

The operator $\mathcal{E}^{\omega,\varphi}(\alpha,\bar{\varphi},\rho)$ is called an \emph{Expansion Oracle} if there exists a non-decreasing continuous function $\Delta:\Real_{++}\rightarrow\Real_{++}$ {such that for any $\bar{\varphi}\geq \varphi^*$ 
satisfies}\footnote{{If $\ubar{\omega}=-\infty$ then the condition should hold true for any $\Delta(\rho)\in(0,\infty)$}} \begin{equation}\label{eq:Delta_demand}
    h(\omega^*-\Delta(\rho))\leq \bar{\varphi}+\rho,\quad \forall \rho:\Delta(\rho)\in (0,\omega^*-\ubar{\omega}],
\end{equation}
and the oracle returns $\beta\in\Real$ satisfying \begin{equation}\label{eq:beta_cond}
    \beta\in \begin{cases}[\alpha+\Delta(\rho),\omega^*], &
    \alpha\leq\omega^*,\; 0<\rho\leq h(\alpha)-\bar{\varphi},\\
    \Real, &\text{otherwise}.
    \end{cases}
\end{equation}
\end{definition}
Note that the expansion oracle maintains a level set defined by $\beta$ that is bounded from above by the optimal value of \eqref{prob:MNP}.

\begin{figure}[b!]
	\centering
	\begin{minipage}{0.9\textwidth}
		\begin{algorithm}[H]
			\caption {ITALEX { - Fixed tolerance (ITALEX-FT)}}\label{Alg:ITALEX_Generalized}
			\SetAlgoLined
			{\bf Input:} $\varepsilon>0$, $\bar{\varphi}\geq \varphi^*$
			, approximation oracle $\mathcal{O}^{\omega,\varphi}$,
			expansion oracle $\mathcal{E}^{\omega,\varphi}$,\\ {$\qquad\quad\tilde{\alpha}_0\leq \bar{\omega}, \tilde{\bx}^0\in\dom(\varphi),\tilde{\bz}^0\in\Lev_{\omega}(\tilde{\alpha}_0)$.}\\
		\For{$k=1,2,...$}
		{
			{$(\rho_k,\tilde{\bx}^{k},\tilde{\bz}^k)=\mathcal{O}^{\omega,\varphi}(\tilde{\bx}^{k-1},\tilde{\bz}^{k-1},\tilde{\alpha}_{k-1},\bar{\varphi},\frac{\varepsilon}{2})$}\\
			\eIf{${\varphi(\tilde{\bx}^{k})+\norm{\tilde{\bx}^k-\tilde{\bz}^k}^2}\leq\bar{\varphi}+\frac{\varepsilon}{2}$}{
				\textbf{return} {$\tilde{\alpha}_{k-1}$ and $(\tilde{\bx}^{k},\tilde{\bz}^k)$}}{
				{$\tilde\alpha_{k}=\mathcal{E}^{\omega,\varphi}(\tilde\alpha_{k-1},\bar{\varphi},\rho_k)$}\\
			}
		}
	\end{algorithm}
\end{minipage}
\begin{minipage}{0.9\textwidth}
	\begin{algorithm}[H]
		\SetAlgoLined
		{Input: $\epsilon_1$, { $\alpha_{0}\leq \omega^*$, $\bx^{0}\in\dom(g)$, $\bz^0\in \Lev_{\omega}(\alpha_{0})$,$\bu^0\in \dom(g)$}\\
			\For {$r=1,2,\ldots$}{
				\textbf{Compute} {$\bu^r=\mathcal{C}(\bu^{r-1},\frac{\epsilon_r}{2})$,  $\bar{\varphi}^r=\varphi(\bu^r)$}.\\
				\uIf{$\varphi(\bx^{r-1})+\norm{\bx^{r-1}+\bz^{r-1}}^2 > \bar{\varphi}^r+\frac{\epsilon_r}{2}$}
				{$(\alpha_{r}, \bx^{r},\bz^r)$ $\leftarrow$ ITALEX-FT($\epsilon_r$, $\bar{\varphi}^r$, $\mathcal{O}^{\omega,\varphi}$,
					$\mathcal{E}^{\omega,\varphi}$, $\alpha_{r-1}$, $\bx^{r-1}$, $\bz^{r-1})$\\}
			\Else{$(\bx^{r},\bz^r)=(\bx^{r-1},\bz^{r-1})$ and $\alpha_{r}=\alpha_{r-1}$}
			Update $\epsilon_{r+1}=\frac{\epsilon_r}{2}$
	}}
	\caption{ITALEX {- Changing tolerance (ITALEX-CT)}}\label{Alg:ITALEX_epsilon_change}
\end{algorithm}
\end{minipage}
\end{figure}

{
We begin by presenting Algorithm~\ref{Alg:ITALEX_Generalized}, the ITALEX scheme for a Fixed Tolerance (ITALEX-FT). For a given $\varepsilon$ and $\bar\varphi>\varphi^*$, ITALEX-FT produces a solution $\bx$ such that  $\varphi(\bx)\leq \bar{\varphi}+\frac{\varepsilon}{2}$ and $\dist(\bx,\Lev_{\omega}(\omega^*))\leq \sqrt{\varepsilon}$. 
We then show that for a given decreasing sequence  $\epsilon^r\rightarrow 0$, if  we can generate  $\bar\varphi^r\leq \varphi^*+\frac{\epsilon^r}{2}$, then nesting Algorithm~\ref{Alg:ITALEX_Generalized} in a loop over $r\in[R]$ results in the ITALEX with Changing Tolerance (ITALEX-CT) scheme, presented in Algorithm~\ref{Alg:ITALEX_epsilon_change}. For any $\varepsilon>0$, ITALEX-CT guarantees that there exists an $R$ such that
\begin{equation}\label{eq:convergence_ITALEX-CT}\varphi(\bx^R)-\varphi^*\leq \varepsilon\quad \text{and}\quad \dist(\bx^R,\Lev_{\omega}(\omega^*))\leq \sqrt{\varepsilon}.\end{equation}
Thus, in order to generate the sequence $\{\bar{\varphi}^r\}_{r\in[R]}$, we assume that we have at our disposal an algorithm $\mathcal{C}$ for solving problem~\eqref{prob:P}, which provided with a starting point $\bu$ and a tolerance $\epsilon$ produces a vector $\bv$ such that $\varphi(\bv)\leq \varphi^*+\epsilon$. Note that algorithm $\mathcal{C}$ can be run in parallel to ITALEX-CT, and since it solves the easier problem~\eqref{prob:P}, it is expected to have a faster convergence rate than ITALEX-CT.}

{The ability of ITALEX-CT to solve problem~\ref{prob:MNP} is summarized in the following result.}
{
\begin{theorem}\label{thm:general_convergence}
Let $\mathcal{O}^{\varphi,\omega}$ be an approximation oracle (as in Definition~\ref{Def:APP_Oracle}), and $\mathcal{E}^{\varphi,\omega}$ be an expansion oracle with a corresponding function $\Delta$ (as in Definition~\ref{Def:EXP_Oracle}). Then,
\begin{enumerate}[label=(\roman*),ref=(\roman*)]
\item\label{thm:ITALEX-FT_convergence} ITALEX-FT (Algorithm~\ref{Alg:ITALEX_Generalized}) is well defined, and terminates after ${\tilde{N}\leq \frac{\omega^*-{\omega(\tilde\bz^0)}}{\Delta(\sfrac{\varepsilon}{4})}}$ iterations with a level set $\tilde{\alpha}_{\tilde{N}-1}$ and a vector $\tilde\bx^{\tilde{N}}$ satisfying: \begin{equation}
    \hskip-7pt \tilde{\alpha}_{\tilde{N}-1}\in[\omega^*-\Delta\left(\frac{\varepsilon}{4}\right), \omega^*],\;\dist(\tilde{\bx}^{\tilde{N}},\Lev_\omega(\omega^*))\leq \sqrt{\varepsilon},
    \;
    \varphi(\tilde{\bx}^{\tilde{N}})\leq \bar{\varphi}+\frac{\varepsilon}{2}.\label{eq:last_iterate}\end{equation}
\item\label{thm:ITALEX-CT_convergence} Let $\{\bx^{r-1}\}_{r\in\mathbb{N}}$ be the iterates produced by ITALEX-CT (Algorithm~\ref{Alg:ITALEX_epsilon_change}). Then, for any $\varepsilon>0$, after $R\leq \lceil\log_2\left(\frac{\epsilon_1}{\varepsilon}\right)\rceil+1$ iterations, we obtain the iterate $\bx^R$ satisfying \eqref{eq:convergence_ITALEX-CT} with a total of $N\leq \frac{\omega^*-{\omega(\bz^0)}}{\Delta(\frac{\varepsilon}{4})}+R$ calls to the approximation oracle. 
\end{enumerate}
\end{theorem}}
{
\begin{proof}
\begin{enumerate}[label=(\roman*),ref=(\roman*)]
\item 
Assume that there exists a finite $\tilde{N}$ such that $\mathcal{E}^{\omega,\varphi}$ outputs a point $\tilde\alpha_{\tilde{N}-1}$ satisfying
\begin{equation}\label{eq:alpha_cond}
\omega^*\geq\tilde\alpha_{\tilde{N}-1}> \omega^*-\Delta\left(\frac{\varepsilon}{4}\right).\end{equation}
Then, from Lemma~\ref{lemma:h_decreasing}, and by the definition of $\Delta$, we obtain that
$$h(\tilde\alpha_{\tilde{N}-1})< h\left(\omega^*-\Delta\left(\frac{\varepsilon}{4}\right)\right)\leq \bar{\varphi}+\frac{\varepsilon}{4}.$$
Therefore, the interval $[\frac{\varepsilon}{4},h(\tilde\alpha_{\tilde{N}-1})-\bar{\varphi}]$ is empty, and at iteration $\tilde{N}$ the approximation oracle $\mathcal{O}^{\omega,\varphi}(\tilde{\bx}^{\tilde{N}-1},\tilde\bz^{\tilde{N}-1},\tilde\alpha_{\tilde{N}-1},\bar{\varphi},\frac{\varepsilon}{2})$ cannot output $\rho\in[\frac{\varepsilon}{4},h(\alpha_{\tilde{N}-1})-\bar{\varphi}]$. Thus, $\mathcal{O}^{\omega,\varphi}$ outputs $\rho=0$ and $(\tilde\bx^{\tilde{N}},\tilde\bz^{\tilde{N}})$ such that $$\varphi(\tilde\bx^{\tilde{N}})+\dist(\tilde\bx^{\tilde{N}},\Lev_{\omega}(\omega^*))^2\leq \varphi(\tilde\bx^{\tilde{N}})+\dist(\tilde\bx^{\tilde{N}},\Lev_{\omega}(\alpha^{\tilde{N}-1}))^2\leq \bar{\varphi}+\frac{\varepsilon}{2}.$$ 
Such $\tilde{\bx}^{\tilde{N}}$ must exist since there exist vectors $\by_1$ and $\by_2$ that are an $\frac{\varepsilon}{4}$ optimal solution to \eqref{def:h} with $\tilde\alpha_{\tilde{N}-1}$, and thus
$$\varphi(\by_1)+\dist(\by_1,\Lev_{\omega}(\tilde\alpha^{\tilde{N}-1}))^2\leq  \varphi(\by_1)+\norm{\by_1-\by_2}^2 \leq h(\tilde\alpha_{\tilde{N}-1})+\frac{\varepsilon}{4}\leq \bar{\varphi}+\frac{\varepsilon}{2}.$$
Finally, using the condition on $\tilde\bx^{\tilde{N}}$ and the non-negativity of the distance function we obtain~\eqref{eq:last_iterate}.

We now show that such a finite $\tilde{N}$ exists.
Assume to the contrary, specifically, that for all $k\leq K\equiv \frac{\omega^*-\omega(\tilde\bz^0)}{\Delta\left(\frac{\varepsilon}{4}\right)}+1$ the algorithm is not terminated.
Since the algorithm is not terminated, it follows from the definition of $\mathcal{O}^{\omega,\varphi}$  that $\frac{\varepsilon}{4} \leq \rho_k\leq h(\tilde\alpha_{k-1})-\bar{\varphi}$, and therefore, by definition of $\mathcal{E}^{\omega,\varphi}$ 
\begin{equation}\label{eq:bound_alpha_gap}
\tilde\alpha_k-\tilde\alpha_{k-1}\geq\Delta\left(\frac{\varepsilon}{4}\right)>0.
\end{equation}
Summing over $k=1,2,\ldots,K-1$, we have 
\begin{equation*}\omega^*-\Delta\left(\frac{\varepsilon}{4}\right)-\omega(\tilde\bz^0)\geq \tilde\alpha_{K-1}-\tilde\alpha_0\geq (K-1)\Delta\left(\frac{\varepsilon}{4}\right)=\omega^*-\omega(\tilde\bz^0),\end{equation*}
where the first inequality follows from the definitions of $\tilde\alpha_0$ and our assumption that $K$ does not satisfy \eqref{eq:alpha_cond}, and the equality follows from the definition of $K$. Thus, arriving at a contradiction. We can therefore deduce that $\tilde{N}\leq K-1$. 
\item
 In any outer-iteration  $r$ of Algorithm~\ref{Alg:ITALEX_epsilon_change}, \ref{thm:ITALEX-FT_convergence} 
guarantees that after $N_r$ calls to the approximation oracle 
$$\dist(\bx^r,\Lev_{\omega}(\omega^*))\leq \sqrt{\epsilon_r},\;\varphi(\bx^r)\leq \bar{\varphi}+\frac{\epsilon_r}{2}\leq \varphi^*+\epsilon_r,$$
with
$$N_r\leq \frac{\omega^*-\omega(\bz^{r-1})}{\Delta\left(\frac{\epsilon_r}{4}\right)}\leq \frac{\alpha_{r}+\Delta\left(\frac{\epsilon_r}{4}\right)-\alpha_{r-1}}{\Delta\left(\frac{\epsilon_r}{4}\right)}=\frac{\alpha_{r}-\alpha_{r-1}}{\Delta\left(\frac{\epsilon_r}{4}\right)}+1,$$
where the second inequality follows from the condition on the outputed $\alpha_r$ and $\omega(\bz^{r-1})\leq \alpha_{r-1}$.
Thus, taking $R$ iterations ensures that $\epsilon_R\leq \varepsilon$, satisfying \eqref{eq:convergence_ITALEX-CT}. To calculate a bound on the total number of calls to the approximation oracle $N$ we use
$$N=\sum_{r=1}^R N_r
\leq \sum_{r=1}^R \left(\frac{\alpha_{r}-\alpha_{r-1}}{\Delta\left(\frac{\epsilon_R}{4}\right)}+1\right)=\frac{\alpha_R-\alpha_0}{\Delta\left(\frac{\epsilon_R}{4}\right)}+R\leq \frac{\omega^*-\omega(\bz^0)}{\Delta\left(\frac{\epsilon_R}{4}\right)}+R$$
where the first inequality is due to  $\Delta(\cdot)$ being nondecreasing, and $\epsilon_r\geq \epsilon_R$ for all $r\leq R$, and the last inequality is a direct result of the definition of $\alpha_0$ and $\alpha_R\leq \omega^*$ from \ref{thm:ITALEX-FT_convergence}.
\end{enumerate}
\end{proof}}

Although the implementation of the general scheme of ITALEX-FT is currently opaque, in the following sections we will show under some mild assumptions a closed formula for our expansion oracle, and two implementations of the approximation oracle using the GCG and PG methods.
\subsection{Expansion oracle construction}\label{subection:Exp}
In this subsection, we will show how to construct an oracle $\mathcal{E}^{\omega,\varphi}$. For this purpose, we need to find a lower bound for $\omega^*$ using our $\rho\leq h(\alpha)-\varphi^*$, and use it later as an update rule for $\alpha_k$. 

Our first assumption restricts the structure of the outer-function $\omega$.
 \begin{assumption}\label{assumption:omega}
   $\omega(\cdot)$ is a convex norm-like function. That is, $\; \omega:\Real^n\rightarrow \Real$ is convex and satisfies the following properties.
   \begin{enumerate}[label=(\alph*),ref=(\alph*)]
   \item \label{ass:omega_compact} For any $\alpha\in\Real$, The level set $\Lev_{\omega}(\alpha)$ is compact.
   \item \label{ass:omega_error_bound} $\omega$ has a {$\kappa$-power} $\gamma$-global error-bound, \ie \begin{equation*}\exists \gamma>0,{\kappa\in(0,2]}:\forall \bx\in \Real^n, \dist(\bx,\Lev_{\omega}(\alpha)){^\kappa}\leq \gamma[\omega(\bx)-\alpha]_+.\end{equation*}
   \end{enumerate}
 \end{assumption}
Assumption \ref{assumption:omega}\ref{ass:omega_compact} is easily satisfied by any closed coercive function, and is significantly weaker than the strong-convexity assumption made in previous works.
Assumption \ref{assumption:omega}\ref{ass:omega_error_bound} is satisfied for functions such as the ones in Table~\ref{tbl:error_bound}.
Specifically, Assumption~\ref{assumption:omega} is holds when $\omega$ is any norm.
 
We are now ready to construct the expansion oracle. 
\begin{theorem}\label{thm:Expansion}
Let Assumption \ref{assumption:omega} hold.
Define $\Delta:\Real_{++}\rightarrow\Real_{++}$ as 
$\Delta(\rho)=\frac{{\rho}^{\sfrac{\kappa}{2}}}{\gamma}.$
Then, $\mathcal{E}^{\omega,\varphi}(\alpha,\bar{\varphi},\rho)=\alpha+\Delta(\rho)$ is an expansion oracle.
\end{theorem}
\begin{proof} 
From the definition of $\Delta(\rho)$, it is easy to see that for any $\rho>0$ it is a nonnegative and increasing function of $\rho$. Moreover, let $\bx^*$ be an optimal solution of problem~\eqref{prob:MNP}. Then, it follows from \eqref{def:h} that for any $\alpha\in [\ubar{\omega},\omega^*]$
\begin{equation}\label{eq:h_connection}
h(\alpha)\leq \varphi(\bx^*)+\dist(\bx^*,\Lev_{\omega}(\alpha))^2 \leq \varphi(\bx^*)+{\left(\gamma(\omega^*-\alpha)\right)^{\sfrac{2}{\kappa}}},\end{equation}
where the last inequality follows from Assumption~\ref{assumption:omega}\ref{ass:omega_error_bound}. 
Thus, for any $\rho$ such that $\Delta(\rho)\in (0,\omega^*-\ubar{\omega}]$ defining $\tilde{\alpha}=\omega^*-\Delta(\rho)$,
and plugging it back into \eqref{eq:h_connection} we immediately obtain property \eqref{eq:Delta_demand}. Moreover, setting $0<\rho\leq h(\alpha)-\bar{\varphi}\leq h(\alpha)-\varphi^*$, combining \eqref{eq:h_connection} with the definition of $\Delta(\cdot)$ implies
$$\Delta(\rho)\leq \frac{({h(\alpha)-\varphi^*})^{\sfrac{\kappa}{2}}}{\gamma}\leq \omega^*-\alpha,$$ which proves inclusion~\eqref{eq:beta_cond}.
\end{proof}

Plugging the definition of the expansion oracle from Theorem~\ref{thm:Expansion} into  Theorem~\ref{thm:general_convergence} results in the following iteration complexity bound on ITALEX-CT.
{
\begin{corollary}\label{cor:ITALEX_convergence_specific}
Let Assumption~\ref{assumption:omega} hold. For any $\varepsilon>0$, after $R\leq \lceil\log_2\left(\frac{\epsilon_1}{\varepsilon}\right)\rceil+1$ iterations of ITALEX-CT, requiring a total number of calls the approximation oracle equal to $N$, where
$$N\leq \frac{2^\kappa\gamma(\omega^*-\omega(\bz^0))}{{\varepsilon}^{\sfrac{\kappa}{2}}}+R,$$
we obtain a solution $\bx^R$ such that \eqref{eq:convergence_ITALEX-CT} is satisfied and additionally
$$
\omega(\bx^R)-\omega^*\leq \ell_{\omega,R_0}\sqrt{\varepsilon}$$ 
where $\bar{\omega}$ be an upper bound on $\omega^*$, and $\ell_{\omega,R_0}$ is the finite Lipschitz constant of $\omega$ over the compact set $$\mathcal{W}^0=\left\{\bx\in\Real^n:\dist(\bx,\Lev_{\omega}(\alpha_0))\leq R_0\equiv (\gamma(\bar{\omega}-\omega(\bz^0)))^{\sfrac{1}{\kappa}}+\sqrt{\varepsilon}
\right\}.$$
\end{corollary}}
\begin{proof}
{The limit on $N$, $\varphi(\bx^R)$, and $\dist\left(\bx^{{R}},\Lev_{\omega}(\omega^*)\right)$ are a direct result of plugging $\Delta(\rho)$ from Theorem~\ref{thm:Expansion} into Theorem~\ref{thm:general_convergence}\ref{thm:ITALEX-CT_convergence}.}
To show the second part,  we will first prove that $\bx^{{R}}$ and $\tilde{\bx}=\proj_{\Lev_{\omega}(\omega^*)}(\bx^{{R}})$ belong to set $\mathcal{W}^0$.
By definition of distance and projection we have that 
\begin{align*}
\dist({\bx}^{{R}},\Lev_{\omega}(\alpha_0))&\leq \norm{\bx^{{R}}-\proj_{\Lev_{\omega}(\alpha_0)}(\tilde{\bx})}\\
&\leq \norm{{\bx}^{{R}}-\tilde{\bx}}+\norm{\tilde{\bx}-\proj_{\Lev_{\omega}(\alpha_0)}(\tilde{\bx})}\\
&= \dist(\bx^{{R}},\Lev_{\omega}(\omega^*))+\dist(\tilde{\bx},\Lev_{\omega}(\alpha_0)).
\end{align*}
Moreover, by Assumption~\ref{assumption:omega}\ref{ass:omega_error_bound},
\begin{align*}
    \dist(\tilde{\bx},\Lev_{\omega}(\alpha_0)){^\kappa}&\leq \gamma(\omega^*-\alpha_0)\leq \gamma(\bar{\omega}-\omega(\bz^0)).
\end{align*}
Since \eqref{eq:last_iterate} from Theorem~\ref{thm:general_convergence} holds, the above inequalities imply that $\bx^{{R}},\tilde{\bx}\in \mathcal{W}^0$.
Since  $\omega$ is convex with $\dom(\omega)=\Real^n$,
\cite[Theorem 7.36]{beck2014introduction}  implies that $\omega$ is Lipschitz continuous over compact sets. By Assumption~\ref{assumption:omega}\ref{ass:omega_compact} we know that $\mathcal{W}^0$ is compact and thus $\ell_{\omega,R_0}$ exists. Thus, applying the Lipschitz property on $\bx^{{R}}$ and $\tilde{\bx}$ together with \eqref{eq:last_iterate} implies that
$$\omega(\bx^{{R}})-\omega^*\leq \ell_{\omega,R_0}\norm{\bx^{{R}}-\tilde{\bx}}{=} \ell_{\omega,R_0}\dist(\bx^{{R}},\Lev_{\omega}(\omega^*))\leq \ell_{\omega,R_0}\sqrt{\varepsilon}.$$ \vskip-20pt
\end{proof}

\subsection{Approximation oracle construction using first-order methods}
In this section, we show that one can construct approximation oracles using first-order methods. We start by presenting a general approach for constructing an approximation oracle based on the notion of an optimality measure. Then, we show how to apply this general approach to specific first-order methods such as GCG and PG. 

We start by making the following standard assumption on the structure of the inner function $\varphi$.
\begin{assumption}\label{assumption:phi} The inner function $\varphi\equiv f+g$ satisfies the following:
    \begin{enumerate}[label=(\alph*)]
        \item $f:\Real^n\rightarrow\Real$ is 
        convex and continuously differentiable with a Lipschitz-continuous gradient with constant $L_f$, \ie \begin{equation*}
        \|\nabla f(\bx)-\nabla f(\mathbf{y})\|\leq L_f\|\bx-\mathbf{y}\|,\; \forall \bx,\by\in\Real^n.\end{equation*}
        \item $g:\Real^n\rightarrow\Real\cup\{\infty\}$  
        is a proper, closed, and convex function.
        \end{enumerate}
\end{assumption}

For ease of notation in the following sections, we define the function $\hat{\varphi}^\alpha:\Real^n\times\Real^n\rightarrow\Real\cup\{\infty\}$ for any $\by=(\by_1,\by_2)$ to be
$$\hat{\varphi}^\alpha(\by)=\varphi(\by_1)+\norm{\by_1-\by_2}^2+\delta_{\Lev_\omega(\alpha)}(\by_2).$$
It is obvious that $\hat{\varphi}^\alpha$ is a composite function. Specifically, $\hat{\varphi}^\alpha=\hat{f}+\hat{g}^\alpha$, where $\hat{f}(\by)=f(\by_1)+\norm{\by_1+\by_2}^2$
is convex and continuously differentiable with $(L_f+2)$-Lipschitz continuous gradient, and 
$\hat{g}^\alpha(\by)=g(\by_1)+\delta_{\Lev_\omega(\alpha)}(\by_2)$.

\subsubsection{A General Approach for Approximation  Oracle Construction}
We start by presenting a general approach for the construction of approximation oracles. This approach is based on the existence of algorithms that satisfy convergence with regard to an optimality measure which we now define.
\begin{definition}\label{def:cont_opt_measure} A continuous optimality measure $\mu^\alpha:\dom(g)\times\Lev_\omega(\alpha)\rightarrow \Real_+$ is a continuous function with the following properties: \begin{enumerate}[label=(\roman*)]
    \item For any $(\by_1,\by_2)\in \dom(g)\times\Lev_{\omega}(\alpha)$ \begin{equation*}
        \mu^\alpha(\by)\geq \varphi(\by_1)+\norm{\by_1-\by_2}^2-h(\alpha)\equiv \hat{\varphi}^\alpha(\by)-h(\alpha).
    \end{equation*}
    \item $\mu^\alpha(\by)=0$ if and only if $\by$ is an optimal solution of problem \eqref{def:h}, and specifically, $\by_2=\proj_{\Lev_{\omega}(\alpha)}(\by_1)$.
\end{enumerate}
\end{definition}

Now assume that we have at our disposal an iterative optimization algorithm $\mathcal{A}$. Given a starting point $\by^{0}$ and a level set parameter $\alpha$, utilizing $\mathcal{A}$ to solve problem~\eqref{def:h} generates iterates $\left\{\left(\by^{j+1},\mu^{\alpha}(\by^j)\right)\right\}_{j\in\mathbb{N}}$ satisfying  
\begin{equation}\label{eq:alg_cond}\lim_{j\rightarrow\infty}\mu^\alpha(\by^j)=0,\end{equation}
with respect to some associated continuous optimality measure $\mu^\alpha$.
Notice that we require algorithm $\mathcal{A}$ to compute $\mu^\alpha(\by^j)$ as part of generating iterate $\by^{j+1}$.
We can now define an approximation oracle that utilises such an algorithm. This approximation oracle is presented in Algorithm~\ref{alg:approx_alg}.

\begin{figure}
\centering
\begin{minipage}{0.9\textwidth}
\begin{algorithm}[H]\caption{A {$\mu^\alpha$} based Approximation Algorithm \label{alg:approx_alg}}
\SetAlgoLined
{\bf Input:} Initial point $\by^0\equiv(\by_1^0,\by_2^0)\in \dom(g)\times\Lev_{\omega}(\alpha)$, $\alpha\leq \omega^*$, $\bar{\varphi}\geq \varphi^*$, $\varepsilon$, and algorithm $\mathcal{A}$ satisfying \eqref{eq:alg_cond}.\\
\For{$j=0,1,2,...$}{
Apply one iteration of $\mathcal{A}$ at point $\by^{j}$ to obtain $\by^{j+1}$ and $\mu^\alpha(\by^j)$.\quad  ($\mathcal{A}$-STEP)\\
\If{$\hat{\varphi}^\alpha(\by^j)-\bar{\varphi}\leq\varepsilon$}{Exit and \textbf{return} $(\rho,\by)=(0,\by^j)$}
\If{$\hat{\varphi}^\alpha(\by^j)-\bar{\varphi}-\mu^\alpha(\by^j)> \frac{\varepsilon}{2}$}{Exit and \textbf{return} $(\rho,\by)=(\hat{\varphi}^\alpha(\by^j)-\bar{\varphi}-\mu^\alpha(\by^j),\by^j)$}
}
\end{algorithm}
\end{minipage}
\end{figure}

We will now prove that Algorithm~\ref{alg:approx_alg} is indeed an approximation algorithm.
\begin{theorem}
Let $\mathcal{A}$ be an algorithm with associated continuous optimization oracle $\mu^\alpha$ satisfying \eqref{eq:alg_cond}. Then, Algorithm~\ref{alg:approx_alg} is a well defined approximation algorithm, which concludes after a finite number of iterations.
\end{theorem}
\begin{proof}
First we will show that the algorithm will terminate after a finite number of iterations. Since $\mu^\alpha(\by^j)\rightarrow 0$, Definition~\ref{def:cont_opt_measure}(ii) implies $\hat{\varphi}^\alpha(\by^j)\rightarrow h(\alpha)$.

Therefore, if there exists a $\delta>0$ such that $h(\alpha)\leq \bar{\varphi}+\varepsilon-\delta$, after some finite number of iterations $J_1$, $\hat{\varphi}^\alpha(\by^{J_1})-h(\alpha)\leq \delta$. Thus, the algorithm will terminate after at most $J_1$ steps with an output $(0,y^{J_1})$ satisfying
$$\varphi(\by^{J_1})\leq \bar{\varphi}+\varepsilon.$$

Assume that the algorithm did not terminate by the first stopping condition. 
Define the function
$$\rho(\by)=\hat{\varphi}^\alpha(\by)-\bar{\varphi}-\mu^\alpha(\by)$$
for any $ \by\in \dom(g)\times \Lev_{\omega}(\alpha)$. Definition~\ref{def:cont_opt_measure}(ii) implies that
\begin{equation*}
     \rho(\by)=\hat{\varphi}^\alpha(\by)-\bar{\varphi}-\mu^\alpha(\by)\leq \hat{\varphi}^\alpha(\by)-\bar{\varphi}-(\hat{\varphi}^\alpha(\by)-h(\alpha))=h(\alpha)-\bar{\varphi}.
 \end{equation*}
Since it did not terminate by the first condition, there are two options: (i) $h(\alpha)\geq \bar{\varphi}+\varepsilon$ (ii) it was terminated by the second conditions for some $J_2\leq J_1$. In case (i), for a large enough $J_2$, we will have that 
$$\mu^{\alpha}(\by^{J_2})< \frac{\varepsilon}{2}\leq h(\alpha)-\bar{\varphi}-\frac{\varepsilon}{2}\leq   \hat{\varphi}^\alpha(\by^{J_2})-\bar{\varphi}-\frac{\varepsilon}{2},$$
where the first inequality follows from \eqref{eq:alg_cond}, and the last inequality stems from the suboptimality of $\by^{J_2}$. In both cases (i) and (ii), ITALEX-FT terminates after a finite number of steps by the second stopping condition, with $\rho=\rho(y^{J_2})\in [\frac{\varepsilon}{2},h(\alpha)-\bar{\varphi}]$.

To conclude, the algorithm is terminated by one of the stopping conditions after a finite number of steps, and satisfying the requirements of the approximation oracle in Definition~\ref{Def:APP_Oracle}. 
\end{proof}

\subsubsection{Convergence rate analysis using sufficient decrease assumption}
In the previous section, we showed that, given a continuous optimality measure $\mu^{\alpha}$ and an iterative algorithm $\mathcal{A}$ for which it converges to $0$, we can construct an approximation oracle. Under some additional assumptions on algorithm $\mathcal{A}$ used in Algorithm~\ref{alg:approx_alg}, we can bound the number of total iterations of the algorithm $\mathcal{A}$ needed during the execution of ITALEX-CT. We will later show that these conditions are satisfied by widely used first-order methods such as GCG and PG.\footnote{In fact, it can be shown that they are satisfied by any PDA method as defined in \cite{PDA}.}
\begin{assumption}\label{assum:sufficient_decrease}
For any $\alpha\leq \omega^*$, and any starting point $\by^0\in\dom(g)\times\Lev_{\omega}(\alpha)$, algorithm $\mathcal{A}$ with optimality measure $\mu^\alpha$ generates a sequence $\newline \Bigl\{(\by^j,\mu^{\alpha}(\by^{j-1}))\Bigr\}_{j\in\mathbb{N}}$ that satisfies
\begin{equation}\label{eq:Sufficient_Decrease}
    \hat{\varphi}^\alpha(\by^{j-1})-\hat{\varphi}^\alpha(\by^{j})\geq \min\{\eta_1\mu^\alpha(\by^{j-1}),{\eta_2\left(\mu^\alpha(\by^{j-1})\right)^2}\},
\end{equation}
for some $\eta_1\in(0,1)$
and $\eta_2\in(0,+\infty]$.
\end{assumption}
Using this sufficient decrease assumption on algorithm $\mathcal{A}$, in the convergence analysis of ITALEX-FT we will analyze the total number of iterations of $\mathcal{A}$ used by ITALEX-CT. 
In order to establish such a result we require the following technical lemma, the proof of which is given in  appendix~\ref{appendix:proof_lemma_decreasing_sequence}.
\begin{lemma}\label{lemma:decreasing_Sequence}
Let \(\{\xi_p\}_{p\in\mathbb{N}}\) be a non-negative sequence, and let $\eta>0$ such that  $\xi_{p+1}\leq\xi_p-\eta \xi_{p+1}^2$ for any $p\in\mathbb{N}$ . Then, for any $p\in \mathbb{N}$ \vskip-5pt
\begin{equation*}
    \xi_p\leq\frac{\max\{\frac{2}{\eta},\xi_1\}}{p}.
\end{equation*}
\end{lemma}
We are now ready to prove the complexity result for {ITALEX-CT} using an approximation oracle utilizing a sub-algorithm $\mathcal{A}$ that satisfies the sufficient decrease property of Assumption \ref{assum:sufficient_decrease}. {For the analysis, we fix an iteration $r$ of ITALEX-CT, and define $M_k$ to be the number of calls to $\mathcal{A}$ during the $k$th iteration of ITALEX-FT.
Additionally, for any $j\in[M_k]$ we also define the vector $\by^{k,j}$ to be the output of he 
$j$th call to algorithm $\mathcal{A}$ during the $k$th iteration of ITALEX-FT, and initialize $\by^{k,0}=(\tilde{\bx}^{k-1},\tilde{\bz}^{k-1})$. Note that $(\tilde{\bx}^{k},\tilde{\bz}^{k})=\by^{k,M_{k}-1}$. Using these definitions we prove the following result.}
\begin{theorem}\label{theorem:Convergence}
Consider the {ITALEX-CT} algorithm ({Algorithm~\ref{Alg:ITALEX_epsilon_change}}) that uses the approximation oracle presented in Algorithm~\ref{alg:approx_alg} with $\mathcal{A}$ satisfying Assumption~\ref{assum:sufficient_decrease}. Then, {for any $\varepsilon>0$, the total number of $\mathcal{A}$ iterations used by ITALEX-CT to produce $\bx^R$ satisfying \eqref{eq:convergence_ITALEX-CT} 
is given by {$\bar{M} \leq K_1+K_2+N$},
where
\begin{align*}
    K_1&=
    \log_{\frac{1}{1-\eta_1}}\left(\min\left\{\frac{\eta_2}{\eta_1},\frac{4}{\epsilon_1}\right\}\left(\varphi(\bx^0)+\norm{\bz^0-\bx^0}^2-\bar{\varphi}_1\right)\right) ,
    \\
    K_2&=
    \frac{32}{\eta_2\varepsilon}+\left(\log_{\frac{1}{1-\eta_1}}\left(9\right)+2\right)
    \left(\left\lceil\log_2\left(\frac{\epsilon_1}{\varepsilon}\right)\right\rceil+1\right),
\end{align*}
$N$ is the total number of calls to the approximation oracle, and $R$ is the total number of iteration of ITALEX-CT with bounds given in Theorem~\ref{thm:general_convergence}}\ref{thm:ITALEX-CT_convergence}.
\end{theorem}
\begin{proof}
{We begin by computing the total number of iterations of $\mathcal{A}$ for some fixed ITALEX-CT iteration $r\in[R]$, which we denote by $\bar{M}_{r}$. Let $N_r$ be the total number of ITALEX-FT iterations, or equivalently, the total number of calls to the approximation oracle 
within the $r$th ITALEX-CT iteration, 
then 
$\bar{M}_{r}=\sum_{k=1}^{N_r}M_k$. Moreover, the stopping condition for ITALEX-FT ensures that in its last ($N_r$) iteration the first stopping criteria of the approximation oracle is used and so the last inner iterate (iterate $M_{N_r}-1$ in ITALEX-FT iteration $N_{r}$) satisfies
\begin{equation}\label{eq:stopping_delta}
\hat{\varphi}^{\tilde\alpha_{N_{{r}}-1}}(\by^{N_{{r}},M_{N_{{r}}}-1})\leq \bar{\varphi_r}+\frac{{\epsilon_r}}{2}.
\end{equation}}
Moreover, for any iteration {$k\in[N_{r}]$} and any $j\in[M_{k}-1]$, since the stopping condition of Algorithm~\ref{alg:approx_alg} was not satisfied in the previous inner iteration 
\begin{equation}\label{eq:Convergence_Theorem_1}
  \mu^{\tilde\alpha_{k-1}}(\by^{k,j-1})\geq \hat{\varphi}^{\tilde\alpha_{k-1}}(\by^{k,j-1})-{\bar{\varphi}_{r}-\frac{{\epsilon_r}}{4}}\geq 0.  
\end{equation}
Furthermore, Assumption~\ref{assum:sufficient_decrease} states that for all  {$k\in[N_{r}]$} and any $j\in[M_{k}]$
\begin{align}
    \hspace{-6pt}\hat{\varphi}^{\tilde\alpha_{k-1}}(\by^{k,j})\hspace{-1pt}\leq\hspace{-1pt} \hat{\varphi}^{\tilde\alpha_{k-1}}(\by^{k,j-1})\hspace{-1pt}-\hspace{-1pt}\min\{\eta_1\mu^{\tilde\alpha_{k-1}}(\by^{k,j-1}),\eta_2\mu^{\tilde\alpha_{k-1}}(\by^{k,j-1})^2\}.\hspace{-4pt}\label{eq:Convergence_Theorem_2}
\end{align}
Defining 
$$d^{k,j}\equiv\hat{\varphi}^{\tilde\alpha_{k-1}}(\by^{k,j})-{\left(\bar{\varphi}_{r}+\frac{\epsilon_r}{4}\right)},\quad {k\in[N_{r}]},j\in\{0,\ldots,M_{k}\}$$ 
we can plug bound \eqref{eq:Convergence_Theorem_1} into \eqref{eq:Convergence_Theorem_2} to obtain for all 
$j\in[M_{k}]$
\begin{equation}\label{eq:Convergence_Theorem_3}
    d^{k,j}\leq d^{k,j-1}-\min\{\eta_1d^{k,j-1},\eta_2(d^{k,j-1})^2\}
\end{equation}
Moreover, since $\tilde\alpha_{k-1}\leq \tilde\alpha_{k}$ and $\by^{k,M_k-1}=\by^{k+1,0}
$, we have $\hat{\varphi}^{\tilde\alpha_{k-1}}(\by^{k,M_k-1})=\hat{\varphi}^{\tilde\alpha_{k}}(\by^{k+1,0})$ and $d^{k,M_k-1}=d^{k+1,0}$.
Thus, we can define $\xi_{0}=d^{1,0}$, and 
\begin{align*}
    \xi_{p}=d^{k,j}, \; p=\sum_{i=1}^{k-1}(M_{i}-1)+j, \; 1\leq j\leq M_{k}-1,\; k=1,\ldots,{N_{r}}
\end{align*}
\noindent\ie\; $\xi_{p}$ is a unified sequence of all $d^{k,j}$, ordered first by the {ITALEX-FT iteration index $k$, and then by the inner iteration index $j$ not accounting for inner iteration $0$ at each $k$.} Thus, \eqref{eq:Convergence_Theorem_3} is equivalent to
\begin{equation}\label{eq:Convergence_Theorem_4}
    \xi_{p}\leq \xi_{p-1}-\min \{\eta_1 \xi_{p-1},\eta_2\xi_{p-1}^2\}.
\end{equation}
By the definition of $d^{k,j}$ {and since $\by^{1,0}=(\tilde\bx^{0},\tilde\bz^{0})=(\bx^{r-1},\bz^{r-1})$,} we have that
\begin{equation}\label{eq:Convergence_Theorem_3a}\xi_{0}=d^{1,0}{= \hat{\varphi}^{\tilde\alpha_{0}}(\by^{1,0})}-{\bar{\varphi}_{r}-\frac{\epsilon_r}{4}=\hat{\varphi}^{\alpha_{r-1}}(\bx^{r-1},\bz^{r-1})-\bar{\varphi}_{r}-\frac{\epsilon_r}{4}}, 
\end{equation}
and from \eqref{eq:stopping_delta} we deduce that $\xi_{p}\leq {\frac{\epsilon_r}{4}}$ for {$p=\bar{M_{r}}-N_{r}$}.\\
We look at the values of $\xi_{p}$ in two stages: (1) $p<{ K_{1,r}}$ in which $\xi_{p}> \frac{\eta_1}{\eta_2}$ (2) {$K_{1,r} \leq p\leq  \bar{M_{r}}-N_{r}$} 
in which $\xi_{p}\leq \frac{\eta_1}{\eta_2}$. 
In stage 1, 
the minimum in \eqref{eq:Convergence_Theorem_4} is attained in the first term, and \eqref{eq:Convergence_Theorem_4} is equivalent to
 $\xi_{p}\leq (1-\eta_1)\xi_{p-1}$.
Thus, the smallest {$K_{1,r}$} such that {$\xi_{K_{1,r}}\leq\max\{\frac{{\epsilon_r}}{4},\frac{\eta_1}{\eta_2}\}$} is bounded by
\begin{equation*}
    {K_{1,r}}=\left\lceil \log_{\frac{1}{1-\eta_1}}\left(\min\left\{\frac{\eta_2}{\eta_1},\frac{4}{{\epsilon_r}}\right\}{\left(\hat{\varphi}^{\alpha_{r-1}}(\bx^{r-1},\bz^{r-1})-\bar{\varphi}_{r}-\frac{\epsilon_r}{4}\right)} \right)  \right\rceil.
\end{equation*}
Let {$K_{2,r}=\bar{M_{{r}}}-N_{{r}}-K_{1,{r}}$}. {$K_{2,r}>0$} only if $\frac{{\epsilon_r}}{4}< \frac{\eta_1}{\eta_2}$. Thus, if {$K_{2,r}>0$},   any {$K_{1,r} \leq p\leq K_{1,r}+K_{2,r}$}, \ie\;  any second stage iteration, satisfies $$\xi_{p+1}\leq \xi_{p}-\min \{\eta_1 \xi_{p},\eta_2\xi_{p}^2\}= \xi_{p}-\eta_2\xi_{p}^2\leq\xi_{p}-\eta_2\xi_{p+1}^2$$
where the equality follows from $\xi_{p}\leq\frac{\eta_1}{\eta_2}$ and \eqref{eq:Convergence_Theorem_4}, and the inequality results from $\xi_{p}$ being nonincreasing.
Thus, by Lemma~\ref{lemma:decreasing_Sequence} \begin{equation}\label{eq:Convergence_Theorem_7}
    \xi_{p}\leq \frac{\max\left\{\frac{2}{\eta_2},\xi_{{K_{1,r}}}\right\}}{p+1-{K_{1,r}}}.
\end{equation} 
{Since 
$\eta_1\in(0,1)$, we obtain 
$\xi_{K_{1,r}}\leq \frac{\eta_1}{\eta_2}\leq \frac{2}{\eta_2}$}, and \eqref{eq:Convergence_Theorem_7} implies that the smallest {$K_{2,r}$} such that {$\xi_{K_{1,r}+K_{2,r}}\leq\frac{{\epsilon_r}}{4}$} satisfies
\begin{equation}\label{eq:Convergence_Theorem_7a}{K_{2,{r}}\leq \left\lceil\frac{4}{\epsilon_r}\cdot\frac{2}{\eta_2}=\frac{8}{{\epsilon_r}\eta_2}\right\rceil}.
\end{equation}
{Finally, we use the fact that for every $r>0$  
 \begin{align}\label{eq:connection_between_r}
    \hspace{-10pt}\hat{\varphi}^{\alpha_{r}}(\bx^{r},\bz^{r})\hspace{-1pt}-\hspace{-1pt}\bar{\varphi}_{r+1}\hspace{-1pt}-\hspace{-1pt}\frac{\epsilon_{r+1}}{4}\hspace{-2pt}=\hspace{-1pt}d^{N_r,M_{N_r}-1}\hspace{-1pt}+\hspace{-2pt}\bar{\varphi}_{r}\hspace{-1pt}-\hspace{-1pt}\bar{\varphi}_{r+1}\hspace{-1pt}+\hspace{-1pt}\frac{\epsilon_r\hspace{-1pt}-\hspace{-1pt}\epsilon_{r+1}}{4}\hspace{-1pt}\leq \hspace{-1pt}\frac{9\epsilon_{r+1}}{4}.  \hspace{-3pt}
\end{align}
where the equality follows from $\alpha_r=\tilde\alpha_{N_{r}-1}$, $(\bx^r,\bz^r)=\by^{N_r,M_{N_r}-1}$, and $2\epsilon_{r+1}=\epsilon_r$, and the inequality follows from $\bar{\varphi}_{r+1}\geq \varphi^*$, $\bar{\varphi}_{r}-\varphi^*\leq \frac{\epsilon_r}{2}$, and $d^{N_r,M_{N_r}-1}\leq \frac{\epsilon_{r}}{2}$ (the stopping criteria of ITALEX-FT).
Thus, combining 
\eqref{eq:Convergence_Theorem_7a}, 
the definition of $\hat{\varphi}^{{\alpha}_r}$, and \eqref{eq:connection_between_r}
we obtain the following bounds. 
\begin{align*}
    K_{1,r}&\hspace{-1pt}\leq\hspace{-1pt} \begin{cases}\left\lceil\log_{\frac{1}{1-\eta_1}}\hspace{-3pt}\left(\min\hspace{-2pt}\left\{\hspace{-1pt}\frac{\eta_2}{\eta_1},\frac{4}{\epsilon_1}\hspace{-2pt}\right\}\left(\varphi(\bx^{0})\hspace{-1pt}+\hspace{-1pt}\norm{\bx^0-\bz^0}^2\hspace{-1pt}-\hspace{-1pt}\bar{\varphi}_{1}\right)\right)\right\rceil
    \hspace{-2pt},&\hspace{-6pt}r=1,\\[0.2em] 
    \left\lceil\log_{\frac{1}{1-\eta_1}}\hspace{-3pt}\left(9\right)\right\rceil, & \hspace{-6pt}r>1,\\
    \end{cases}\\
    K_{2,r}\hspace{-1pt}&\leq \hspace{-1pt}
    \begin{cases}0 ,\qquad\qquad\qquad\qquad\qquad\qquad\qquad\qquad\qquad\qquad\qquad&\quad\;\eta_1\in\left(0,\frac{\eta_2\epsilon_r}{4}\right],\\
    \left\lceil\frac{8}{\eta_2\epsilon_r}\right\rceil ,& \quad\;\eta_1\in(\frac{\eta_2\epsilon_r}{4},1).\\
    \end{cases}
\end{align*}}
{
Recall that 
\begin{equation}\label{eq:barM}\bar{M}=\sum_{r=1}^R\bar{M}_r=\sum_{r=1}^R (K_{1,r}+K_{2,r}+N_r)=\sum_{r=1}^R K_{1,r}+\sum_{r=1}^RK_{2,r}+N,\end{equation}
where $N$ is given by Theorem~\ref{thm:general_convergence}\ref{thm:ITALEX-CT_convergence}. Thus, it is left to compute a bound on the sum of $K_{1,r}$ and $K_{2,r}$ over $r$. 
We define $K_1$ and $K_2$ as follows:
\begin{align*}
K_{1,1}\hspace{-1pt}-\hspace{-1pt}1&\hspace{-1pt}\leq \hspace{-1pt} \log_{\frac{1}{1-\eta_1}}\hspace{-3pt}
\left(\hspace{-2pt}\min\hspace{-2pt}\left\{\hspace{-1pt}\frac{\eta_2}{\eta_1},\frac{4}{\epsilon_1}\hspace{-2pt}\right\}
\hspace{-2pt}\left(\varphi(\bx^{0})\hspace{-1pt}+\hspace{-1pt}\norm{\bx^0-\bz^0}^2\hspace{-3pt}-\hspace{-1pt}\bar{\varphi}_{1}\hspace{-1pt}\right)
\right)
\hspace{-2pt}\equiv K_1\\
\sum_{r=2}^R K_{1,r}\hspace{-1pt} + \hspace{-1pt}\sum_{r=1}^R K_{2,r}\hspace{-1pt}+\hspace{-1pt}1&\hspace{-1pt}\leq \hspace{-1pt} R(\log_{\frac{1}{1-\eta_1}}\hspace{-3pt}\left(9\right)+1)+\sum_{r=1}^R\left(\frac{8}{\eta_2\epsilon_r}+1\right)\\
&\leq R(\log_{\frac{1}{1-\eta_1}}\hspace{-3pt}\left(9\right)+2)+\sum_{r=1}^R\frac{8}{\epsilon_R2^{R-r}\eta_2}\\
&\leq \left(\left\lceil\log_2\left(\frac{\epsilon_1}{\varepsilon}\right)\right\rceil+1\right) (\log_{\frac{1}{1-\eta_1}}\hspace{-3pt}\left(9\right)+2)+\frac{32}{\varepsilon \eta_2}\equiv K_2.
\end{align*}
where the last inequality follows from plugging the bound on $R$ from Theorem~\ref{thm:general_convergence}\ref{thm:ITALEX-CT_convergence} and the fact $\epsilon_R<\varepsilon\leq \epsilon_{R-1}=2\epsilon_R$. Plugging the bounds of $K_1$, $K_2$ into \eqref{eq:barM} obtains the desired result.
}
\end{proof}

\subsubsection{Implementation using first-order methods}
In the two previous sections, we presented a general approximation oracle that utilizes a sufficient descent iterative algorithm $\mathcal{A}$ and proved the iteration complexity of using such an oracle within ITALEX-CT.
We now discuss two specific first-order algorithms that can be used as $\mathcal{A}$: The GCG and the PG.

\paragraph{\bf Implementation via GCG.\\}
In order to prove that GCG can be used in the approximation oracle, we need to define its associated continuous optimality measure and show that GCG satisfies Assumption~\ref{assum:sufficient_decrease} with respect to this measure. Throughout this section we also assume that $\dom(g)$ is compact.

We start by defining the optimality measure $S^{\alpha}$ associated with GCG for solving problem~\eqref{def:h}, by adapting definitions and the lemmas in Section~\ref{section:preliminaries} from problem~\eqref{prob:P} to problem~\eqref{def:h}.
Thus, for any $\alpha\geq \ubar{\omega}$, we define
\begin{equation}\label{def:D_alpha}D_\alpha\geq \mathcal{D}_{\dom({g})}+\mathcal{D}_{\Lev_{\omega}(\alpha)}\geq  \mathcal{D}_{\dom(\hat{g}^\alpha)},\end{equation} an upper bound on the sum of the diameter of $\dom(g)$ and the diameter of the $\alpha$-Level set of $\omega$, which is finite under assumption \ref{assumption:omega}\ref{ass:omega_compact}.
Using these notations we adapt Definition \ref{S_def} to function $\hat{\varphi}^\alpha$.
The $\alpha$-surrogate optimality gap is denoted as
\begin{equation*}
    S^\alpha(\by)\equiv\left\langle \nabla \hat{f}(\by), \by-\mathbf{p}^\alpha(\by)\right\rangle +\hat{g}^{\alpha}(\by)-\hat{g}^\alpha(\mathbf{p}^\alpha (\bx)) , 
\end{equation*}
where
\begin{align*}
    \mathbf{p}^\alpha (\by)&\in \argmin_{\mathbf{p}\in \Real^n\times\Real^n }\left\{\langle \nabla \hat{f} (\by), \mathbf{p}\rangle  +\hat{g}^\alpha(\mathbf{p}) \right\}.
\end{align*}
Note that $\mathbf{p}^\alpha (\by)=(\bp^\alpha (\by)_1,\mathbf{p}^\alpha (\by)_2)$ such that
\begin{align*}
    \bp^\alpha (\by)_1&\in \argmin_{\bp_1\in\Real^n}\left\{\langle \nabla f(\by_1)+2(\by_1-\by_2),\bp_1\rangle+g(\bp_1)\right\},\\
    \bp^\alpha (\by)_2&\in \argmin_{\bp_2\in\Lev_{\omega}(\alpha)}\left\{\langle 2(\by_2-\by_1),\bp_2\rangle\right\}.
\end{align*}
Thus, the GCG can be applied to function $\hat{\varphi}^\alpha$ if it can be applied to $\varphi$ and there exists a simple linear oracle over $\Lev_{\omega}(\alpha)$.
Using these definitions, we obtain the following corollary of Lemma~\ref{S_lemmas} and Lemma~\ref{S_lemmas_sufficient_decrease} from Section~\ref{section:preliminaries}. \begin{corollary}\label{S_alpha_Optimality_measure}
 $S^\alpha(\cdot)$ is a continuous optimality measure for problem~\eqref{def:h} associated with GCG {with adaptive stepsize}.
Moreover, GCG satisfies Assumption~\ref{assum:sufficient_decrease} with respect to $S^\alpha(\cdot)$, \ie\; for any $\alpha\geq\ubar{\omega}$ and for any $\by^0\in\dom(\hat{g}^\alpha)\equiv \dom(g)\times\Lev_{\omega}(\alpha)$ the $j$th iteration of GCG satisfies
 \begin{equation}\label{eq:suff_decr}
        \hat{\varphi}^\alpha(\by^{j})-\hat{\varphi}^\alpha(\by^{j+1})\geq \frac{1}{2}\min\left\{S^{\alpha}(\by^{j}),\frac{S^{\alpha}(\by^{j})^2}{(L_f+2) D_{\alpha}^2}\right\},
    \end{equation}
 \end{corollary}
\begin{proof}
 Let $\alpha\geq \ubar{\omega}$,
 for any $\bu\in \dom(\hat{g}^\alpha)=\dom(g)\times\Lev_{\omega}(\alpha)$, which is compact due to  Assumption~\ref{assumption:omega}\ref{ass:omega_compact}. Therefore, by Lemma~\ref{S_lemmas} we obtain
 \begin{equation*}
     S^\alpha(\bu)\geq \hat{\varphi}^\alpha(\bu)-h(\alpha).
 \end{equation*}
Moreover, $S^\alpha(\bu)=0$ if and only if $\bu\in\argmin_{\by\in \Real^n\times\Real^n}\{\hat{\varphi}^\alpha(\by)\}$ by \cite[Theorem 13.6]{FO_Book}. Thus, $S^\alpha(\cdot)$ is a continuous optimality measure.
Furthermore, by \cite[Theorem 13.9]{FO_Book}, applying GCG to problem~\eqref{def:h} guarantees that $$\lim_{j\rightarrow \infty}S^\alpha(\by^j)=0.$$
Finally, since $D_\alpha$ is greater or equal to the diameter of $\dom(\hat{g}^\alpha)$, by Lemma~\ref{S_lemmas_sufficient_decrease} we obtain \eqref{eq:suff_decr}.
\end{proof}
Thus, we can utilize the GCG as our algorithm $\mathcal{A}$, with iterations defined in Algorithm \ref{Alg:GCG_Oracle}. A straight forward result of Theorem~\ref{theorem:Convergence} for the specific implementation of Algorithm~\ref{Alg:GCG_Oracle} as ($\mathcal{A}-STEP$) in the approximation oracle follows.

\begin{figure}[ht]
\centering
\begin{minipage}{0.9\textwidth}
\begin{algorithm}[H]
\SetAlgoLined
\textbf{Input:} $\by^j\in\dom(g)\times\Lev_\omega(\alpha)$, $\alpha$\\
\(\bp^j=\argmin_{\bu\in\Real^n\times\Real^n}\{\langle \nabla \hat{f}(\by^j),\bu\rangle + \hat{g}^\alpha(\bu)\}\)\\
\(S^\alpha(\by^j)=\langle \nabla \hat{f}(\by^j),\by^j-\bp^j\rangle+\hat{g}^\alpha(\by^j)-\hat{g}^\alpha(\bp^j)\)\\
\(\eta_j=\min\{1,\frac{S^\alpha(\by^j)}{(L_f+2)\|\by^j-\bp^j\|}\}\)\\
\(\by^{j+1}=\by^j+\eta_j(\bp^j-\by^j)\)\\
\textbf {Output:}  $\by^{j+1},S^\alpha(\by^j)$
\caption{($\mathcal{A}$-STEP):Generalized Conditional Gradient Step}\label{Alg:GCG_Oracle}
\end{algorithm}
\end{minipage}
\end{figure}

\begin{corollary}\label{cor:convergence_GCG}
Let Assumptions~\ref{assumption:omega} and \ref{assumption:phi} hold, and $\dom(g)$ be compact. 
Let the approximation oracle $\mathcal{O}^{\varphi,\omega}$ be Algorithm~\ref{alg:approx_alg} using Algorithm~\ref{Alg:GCG_Oracle} as an iteration of $\mathcal{A}$, and let the expansion oracle $\mathcal{E}^{\varphi,\omega}$ be the one defined in Theorem~\ref{thm:Expansion}.
Then, {for any $\varepsilon>0$, ITALEX-CT} 
requires at most $K_1+K_2+N$ iterations of GCG, where
{
\begin{align*}
     K_1&=\log_{2}\left(\min\left\{\frac{2}{2 (L_f+2) D_{\omega^*}^2},\frac{4}{\epsilon_1}\right\}\left(\varphi(\bx^0)+\norm{\bz^0-\bx^0}^2-\bar{\varphi}_1\right)\right), \\
     K_2&=\frac{64 (L_f+2) D_{\omega^*}^2}{\varepsilon}+\left(\log_{2}\left(9\right)+2\right)\left(\left\lceil\log_2\left(\frac{\epsilon_1}{\varepsilon}\right)\right\rceil+1\right),
     \\ N&=\left\lceil \frac{2^\kappa\gamma(\omega^*-\omega(\bz^0))}{\varepsilon^{\sfrac{\kappa}{2}}}\right\rceil+\left\lceil\log_2\left(\frac{\epsilon_1}{\varepsilon}\right)\right\rceil+1.
\end{align*}}
\end{corollary}
\begin{proof}
Since at every call to the approximation oracle $\alpha_k\leq\omega^*$, $D_{\alpha_k}\leq D_{\omega^*}$ and it follows from Corollary~\ref{S_alpha_Optimality_measure} that Assumption~\ref{assum:sufficient_decrease} holds with \begin{equation}\label{eq:cor_GCG1}
\eta_1=\frac{1}{2},\quad \eta_2=\frac{1}{2 (L_f+2) D_{\omega^*}^2}.
\end{equation}
Plugging \eqref{eq:cor_GCG1} and Corollary~\ref{cor:ITALEX_convergence_specific} in Theorem~\ref{theorem:Convergence} gives the desired result.
\end{proof}
\begin{remark}
Algorithm~\ref{Alg:GCG_Oracle} can be replaced by a variation of GCG with backtracking \cite{pedregosa2020linearly}. This algorithm does not require knowledge of a global Lipschitz constant $L_f$ and yields similar convergence guarantees.
\end{remark}
 
\paragraph{\bf Implementation via PG.\\}
Similarly to the case of GCG,  in order to use PG in the approximation oracle, we need to define its associated continuous optimality measure and show that PG satisfies Assumption~\ref{assum:sufficient_decrease} with respect to this measure. Moreover, we require that this continuous optimality measure would be obtained directly from the PG iterations with no additional computational cost.

Therefore, we first need to define the basic iteration of the PG algorithm for problem~\eqref{def:h}. For any $\alpha\geq\ubar{\omega}$, we define the $\alpha$-proximal gradient mapping with regard to function $\hat{\varphi}^\alpha$ at point $\by$
as:
\begin{align*}
    T^\alpha(\by)&\equiv \prox_{\frac{1}{L_f+2}\hat{g}^\alpha}\left(\by-\frac{1}{L_f+2}\nabla \hat{f}(\by)\right).
\end{align*}
Moreover, it is evident that $T^\alpha(\by)=(T^\alpha(\by)_1,T^\alpha(\by)_2)$ where
\begin{align*}
T^\alpha(\by)_1&=\prox_{\frac{1}{L_f+2}{g}}\left(\by_1-\frac{1}{L_f+2}\left(\nabla {f}(\by_1)+2(\by_1-\by_2)\right)\right),\\
T^\alpha(\by)_2&=\proj_{\Lev_{\omega}(\alpha)}\left(\frac{L_f\by_2+2\by_1}{L_f+2}\right),
\end{align*}
and thus can easily be computed if $g$ is a prox-friendly function and the projection onto $\Lev_{\omega}(\alpha)$ is simple.

Since computing the optimality measure $S^\alpha(\by)$ is the same as computing the GCG step, we need an alternative optimality measure based on $T^\alpha$. 
To do this, we first show that for any $\by\in \dom(g)\times\Lev_{\omega}(\alpha)$ the diameter of $\Lev_{\hat{\varphi}^\alpha}(\hat{\varphi}^\alpha(\by))$ is bounded.
 \begin{lemma}\label{lemma:D_alpha}
 Let Assumptions~\ref{assumption:omega} and \ref{assumption:phi} hold. Let {$\bar{\varphi}\leq \varphi^*+\varepsilon/2$, let} $\alpha\in[\ubar{\omega}, \omega^*]$, and let $\by\in \dom(g)\times\Lev_{\omega}(\alpha)$. Then
 $$\mathcal{D}_{\Lev_{\hat{\varphi}^\alpha}(\hat{\varphi}^\alpha(\by))}\leq \tilde{D}_\alpha(\by)\equiv \min\left\{D_\alpha,\sqrt{6\left(\hat{\varphi}^\alpha(\by)-\bar{\varphi}+\frac{\varepsilon}{2}\right)+4\mathcal{D}_{\Lev_{\omega}(\alpha)}^2}\right\},$$
where $D_\alpha$ is defined in ~\eqref{def:D_alpha}.
 \end{lemma}
 \begin{proof}
 Let $\bu,\bv\in \Lev_{\hat{\varphi}^\alpha}(\hat{\varphi}^\alpha(\by))$. Then, it follows from $\hat{\varphi}^\alpha(\by)<\infty$  that  $\bu_2,\bv_2\in\Lev_{\omega}(\alpha)$ and $\bu_1,\bv_1\in\dom(g)$, implying that 
\begin{equation}\label{eq:diam_u2v2}
\norm{\bu_2-\bv_2}\leq \mathcal{D}_{\Lev_\omega(\alpha)}
\end{equation} and
 $\norm{\bu-\bv}\leq D_\alpha$. Moreover, by the definition of $\hat{\varphi}^\alpha$
 \begin{align}
\norm{\bu_1-\bv_1}^2&\leq   3(\norm{\bu_1-\bu_2}^2+\norm{\bv_1-\bv_2}^2+\norm{\bu_2-\bv_2}^2) \nonumber \\
&\leq 3\left({\hat{\varphi}^\alpha(\bu)-\varphi(\bu_1)}+{\hat{\varphi}^\alpha(\bv)-\varphi(\bv_1)}+\mathcal{D}_{\Lev_{\omega}(\alpha)}^2\right)\nonumber\\
&\leq 6(\hat{\varphi}^\alpha(\by)-{\varphi^*})+3\mathcal{D}_{\Lev_{\omega}(\alpha)}^2,\label{eq:diam_u1v1}
 \end{align}
 where the second inequality follows from the definition of $\hat{\varphi}^\alpha$ and \eqref{eq:diam_u2v2}, and the last inequality follows from $\varphi(\bu_1),\varphi(\bv_1)\geq \varphi^*$.
 Combining the \eqref{eq:diam_u2v2}, \eqref{eq:diam_u1v1}, and $\bar{\varphi}\leq \varphi^*+\varepsilon/2$ achieves the desired result.
 \end{proof}

Using $\tilde{D}_{\alpha}(\by)$ defined above, and defining
\begin{equation*}
    \tilde{S}^\alpha(\by)=2\max\left\{\hat{\varphi}^\alpha(\by)-\hat{\varphi}^\alpha(T^\alpha(\by)),\tilde{D}_\alpha(\by)\sqrt{\frac{L_f+2}{2}(\hat{\varphi}^\alpha(\by)-\hat{\varphi}^\alpha(T^\alpha(\by)))}\right\},
\end{equation*}
we immediately obtain the following implication of Lemma~\ref{corollary:proximal_bound}. 
\begin{corollary}\label{cor:S_tilde_alpha_optimality}
 $\tilde{S}^\alpha(\cdot)$ is a continuous optimality measure associate with PG for problem~\eqref{def:h}.\end{corollary} \begin{proof}
Let $\alpha\in\Real$, and let $\by\in \dom(g)\times\Lev_{\omega(\alpha)}$, Then. it follows from Lemma~\ref{corollary:proximal_bound}, that replacing $\varphi$ by $\hat{\varphi}^{\alpha}$ we obtain \begin{equation*}
     \tilde{S}^\alpha(\by)\geq \hat{\varphi}^\alpha(\by)-h(\alpha),
 \end{equation*}
Thus, satisfying the first property of definition \ref{def:cont_opt_measure}. Moreover, it follows from \cite[Corollary 10.8]{FO_Book} that $\hat{\varphi}^\alpha(\by)-\hat{\varphi}^\alpha\left(T^\alpha(\by)\right)\geq 0$ for all $\by$, and is equal to $0$ if and only if $\by$ is optimal.
Thus, $\tilde{S}^\alpha(\by)$ satisfies the second condition of definition \ref{def:cont_opt_measure}. Notice that since the sequence $\varphi(\by^j)$ generated by PG converges, the sequence $\varphi(\by^j)-\varphi\left(T^\alpha(\by^j)\right)$ converges to 0, and therefore so does $\tilde{S}^\alpha(\by^j)$.
\end{proof}
We can now consider the implementation of the $(\mathcal{A}-STEP)$ in the approximation oracle using the PG step presented in Algorithm~\ref{Alg:PG_Oracle}. The following result is a direct implication of applying  Theorem~\ref{theorem:Convergence} to this implementation.
\begin{figure}[ht]
\centering
\begin{minipage}{0.9\textwidth}
 \begin{algorithm}[H]
\SetAlgoLined
\textbf{Input:} $\by^j\in\dom(g)\times\Lev_\omega(\alpha)$, $\alpha$\\
\(\by^{j+1}=\prox_{\hat{g}^\alpha}\left(\by^j-\frac{1}{L_{f}+2}\nabla \hat{f}(\by^j)\right)\)\\
Set
$\tilde{D}_\alpha(\by^j)=\min\left\{D_\alpha,\sqrt{6\left(\hat{\varphi}^\alpha(\by^j)-\bar{\varphi}_\alpha+\frac{\varepsilon}{2}\right)+4\mathcal{D}_{\Lev_{\omega}(\alpha)}^2}\right\}$\\
Set $\zeta_j=\hat{\varphi}^\alpha(\by^{j+1})-\hat{\varphi}^\alpha(\by^j)$\\
Calculate \(\tilde{S}^\alpha(\by^j)=2\max\left\{\zeta_j,\sqrt{\frac{L_{f}+2}{2}\tilde{D}_\alpha(\by^j)^2\zeta_j}\right\}\)\\
\textbf{Output:}  $\by^{j+1},\tilde{S}^\alpha(\by^j)$
\caption{($\mathcal{A}$-STEP):Proximal Gradient Step}\label{Alg:PG_Oracle}
\end{algorithm}
\end{minipage}
\end{figure}
\begin{corollary}\label{cor:PG_Convergence_rate}
Let Assumptions~\ref{assumption:omega} and \ref{assumption:phi} hold. 
Let the approximation oracle $\mathcal{O}^{\varphi,\omega}$ be Algorithm~\ref{alg:approx_alg} using Algorithm~\ref{Alg:PG_Oracle} as an iteration of $\mathcal{A}$, and let the expansion oracle $\mathcal{E}^{\varphi,\omega}$ be the one defined in Theorem~\ref{thm:Expansion}.
Then, {for any $\varepsilon>0$, ITALEX-CT}
requires at most $K_1+K_2+N$ iterations of PG, where
{
\begin{align*}
     K_1&=\log_{2}\left(\min\left\{\frac{2}{2 (L_f+2) \hat{D}_{0}^2},\frac{4}{\epsilon_1}\right\}\left(\varphi(\bx^0)+\norm{\bz^0-\bx^0}^2-\bar{\varphi}_1\right)\right), \\
     K_2&=\frac{64 (L_f+2) \hat{D}_{0}^2}{\varepsilon}+\left(\log_{2}\left(9\right)+2\right)\left(\left\lceil\log_2\left(\frac{\epsilon_1}{\varepsilon}\right)\right\rceil+1\right),
     \\ N&=\left\lceil \frac{2^\kappa\gamma(\omega^*-\omega(\bz^0))}{\varepsilon^{\sfrac{\kappa}{2}}}\right\rceil+\left\lceil\log_2\left(\frac{\epsilon_1}{\varepsilon}\right)\right\rceil+1,
\end{align*}
with $\Delta_0=\varphi(\bx^0)+\norm{\bx^0-\bz^0}^2-\varphi^*+\frac{\epsilon_1}{2}$ 
and $\hat{D}_0^2=\min\left\{6\Delta_0+4\mathcal{D}_{\Lev_{\omega}(\omega)^*}^2,D_{\omega^*}^2\right\}$.}
\end{corollary}
\begin{proof}
By the definition of $\tilde{S}^\alpha$, either 
\begin{equation}\label{eq:prox_cor_1}
     \hat{\varphi}^\alpha(\by)-\hat{\varphi}^\alpha\left(T^\alpha(\by)\right)= \frac{1}{2}\tilde{S}^\alpha(\by),
\end{equation}
or, $\tilde{S}^\alpha(\by)=\tilde{D}_\alpha(\by)\sqrt{2(L_f+2)\left(\hat{\varphi}^\alpha(\by)-\hat{\varphi}^\alpha\left(T^\alpha(\by)\right)\right)}$, implying \begin{equation}\label{eq:prox_cor_2}
     \hat{\varphi}^\alpha(\by)-\hat{\varphi}^\alpha\left(T^\alpha(\by)\right)=\frac{\tilde{S}^\alpha(\by)^2}{2(L_f+2) \tilde{D}_\alpha(\by)^2}.
\end{equation}
{We fix the iteration $r\in[R]$ of ITALEX-CT. The sequence of $\{\tilde\alpha_{k-1}\}_{k\in\mathbb{N}}$ in ITALEX-FT is non-decreasing,
$\hat{\varphi}^{\tilde\alpha_{k}}(\by)\geq \hat{\varphi}^{\tilde\alpha_{k+1}}(\by)$ for any $\by\in\Real^{2n}$. Thus, the sequence {$\{\hat{\varphi}^{\tilde\alpha_{k-1}}((\tilde\bx^{k-1},\tilde\bz^{k-1}))\}_{k\in\mathbb{N}}$} is nonincreasing and bounded from above by {$\hat{\varphi}^{\tilde\alpha_0}((\tilde\bx^0,\tilde\bz^0))=\hat{\varphi}^{\alpha_{r-1}}((\bx^{r-1},\bz^{r-1}))\leq \hat{\varphi}^{\alpha_{0}}((\bx^{0},\bz^{0}))=v_0.$ }%
 Moreover, since ${\tilde\alpha_{k}}\leq \omega^*$, {$\mathcal{D}_{\Lev_{\omega}(\tilde{\alpha_{k}})}\leq \mathcal{D}_{\Lev_{\omega}(\omega^*)}$} and ${{D}_{\tilde\alpha_{k}}}\leq {D}_{\omega^*}$.} {Combining this with $\bar{\varphi}\equiv\bar{\varphi}^{r}\geq \varphi^*$ and $\varepsilon\leq \epsilon_1$, we obtain that $\hat{D}_0\geq \tilde{D}_{\tilde\alpha_{k-1}}(\by)$  for any $\by\in \Lev_{\hat{\varphi}^{\tilde\alpha_{k-1}}}(v_0)$.
Combining \eqref{eq:prox_cor_1}, \eqref{eq:prox_cor_2}, and this upper bound $\hat{D}_0$, implies that for 
any ITALEX-FT iterate $k$ and PG iterate $j$ it holds that
\begin{equation*}
\hat{\varphi}^{\tilde\alpha_{k-1}}(\by^{k,j})-\hat{\varphi}^{\tilde\alpha_{k-1}}\left(\by^{k,j+1}\right)\geq \min\left\{\frac{1}{2}\tilde{S}^{\tilde{\alpha}_{k-1}}(\by^{k,j}),\frac{\tilde{S}^{\tilde\alpha_{k-1}}(\by^{k,j})^2}{2(L_f+2) \hat{D}_0^2}\right\}.
\end{equation*}}
Hence,  Assumption~\ref{assum:sufficient_decrease} holds with $\eta_1=\frac{1}{2},\eta_2=\frac{1}{2(L_f+2) \hat{D}_0^2}$. Plugging $N$ from Corollary~\ref{cor:ITALEX_convergence_specific} in Theorem~\ref{theorem:Convergence} we obtain the desired  result.
\end{proof}
\begin{remark}
Algorithm~\ref{Alg:PG_Oracle} can be replaced with a PG algorithm in which the step-size is determined by a backtracking procedure, using a local Lipschitz constant of the gradient \cite[Section 10.4.3]{FO_Book}. This algorithm does not require knowledge of a global Lipschitz constant $L_f$ and admits similar convergence guarantees. \end{remark}
\begin{remark} 
A similar proof can be used for any nonincreasing $\frac{1}{\beta}$-PDA method as described in \cite{PDA} by computing $\tilde{S}^\alpha$ using the PDA method and multiplying it by $\beta$ to get an optimality measure.
\end{remark}

\section{Special case: smooth inner function}
In this section, we address the special case where $g=0$, that is $\varphi\equiv f$ is a smooth function. We show that in this case, a variation of ITALEX-CT presented in Algorithm~\ref{Alg:ITALEX_Generalized} can be used to get a slightly stronger result, guarantying that for any iteration $k$ the iterator $\bx^{k}$ is super-optimal, \ie $\;\omega(\bx^k)\leq \omega^*$. 

We start by redefining function $H$ and $h$ from Section~\ref{subsec:general_scheme} as
$$H(\bx,\alpha)=f(\bx)+\delta_{\epi(\omega)}(\bx,\alpha),\; h(\alpha)=\min_{\bx\in\Real^n} H(\bx,\alpha), $$
consequently eliminating the need for variables $\bz$ (or equivalently restricting $\bz=\bx$). We still use the same approach of finding the minimal $\alpha$ for which $h(\alpha)\leq \varphi^*$. Thus, as before, we are required to construct an approximation oracle and an expansion oracle satisfying Definitions \ref{Def:APP_Oracle} and \ref{Def:EXP_Oracle}, respectively.

Note that due to the new definition of $h$ the expansion oracle defined in Theorem~\ref{thm:Expansion} is no longer valid. Thus, we present a new expansion oracle, which also relies on the structure of $\varphi$. This expansion oracle requires knowledge of the \emph{global} Lipschitz constant $L_f$ of the gradient of $f$. 
\begin{proposition}\label{prop:expansion_new}
Let Assumption \ref{assumption:omega} hold, and let Assumption \ref{assumption:phi} hold with $g=0$.
Define $\Delta:\Real_{++}\rightarrow\Real_{++}$ as $\Delta(\rho)=\frac{1}{\gamma}{\left(\frac{2\rho}{L_f}\right)^{\sfrac{\kappa}{2}}}$. Then, 
$\mathcal{E}^{\omega,\varphi}(\alpha,\bar{\varphi},\rho)=\alpha+\Delta(\rho)$
is an expansion oracle.
\end{proposition}
\begin{proof}
From the definition of $\Delta(\rho)$, it is easy to see that for any $\rho>0$ it is a nonnegative and increasing function of $\rho$. Moreover, let $\bx^*$ be an the optimal solution of problem~\eqref{prob:MNP}, let $\alpha\in [\ubar{\omega},\omega^*]$, and let  $\tilde{\bx}=P_{\Lev_\omega(\alpha)}(\bx^*)$. Then, 
\begin{align}
h(\alpha)\leq \varphi(\tilde{\bx})=f(\tilde{\bx})&\leq f(\bx^*)+\langle\nabla f(\bx^*),\tilde{\bx}-\bx^*\rangle +\frac{L_f}{2}\norm{\tilde{\bx}-\bx^*}^2\nonumber\\
&\leq \varphi^*+\frac{L_f}{2}\dist(\bx^*,\Lev_{\omega}(\alpha))^2\nonumber\leq \varphi^*+\frac{L_f}{2}{\left(\gamma(\omega^*-\alpha)\right)^{\sfrac{2}{\kappa}}}\label{eq:omega_h_connect2}
\end{align}
where the first inequality is due to the definition of $\tilde{\bx}$ and $h(\alpha)$, the second inequality is due to the Descent Lemma \cite[Lemma 1.2.3]{Convex_Opt_Book}, the third inequality follows from the optimality of $\bx^*$ implying that $\nabla f(\bx^*)=0$ and $f(\bx^*)=\varphi^*$, and the last inequality is due to Assumption~\ref{assumption:omega}.
The rest of the proof is identical to that of Theorem~\ref{thm:Expansion}.
\end{proof}

Redefining $\hat{\varphi}^\alpha$ as
$$\hat{\varphi}^\alpha(\bx)=f(\bx)+\delta_{\Lev_\omega(\alpha)}(\bx),$$
results again in a composite function with compact domain with diameter $\mathcal{D}_{\Lev_\omega(\alpha)}$. Therefore, we can use approximation oracles based on GCG or PG, as before, with appropriate optimality measures. Thus, the following result follows directly from Theorems~\ref{thm:general_convergence} and \ref{theorem:Convergence}, and the fact that   $\omega(\bx^k)\leq \alpha_k\leq \omega^*$ at each outer iteration $k$.
The proof is similar to the proofs of {Theorem \ref{theorem:Convergence},} Corollaries \ref{cor:ITALEX_convergence_specific}, \ref{cor:convergence_GCG} and \ref{cor:PG_Convergence_rate}, and thus is not presented.
\begin{corollary}
Let Assumptions~\ref{assumption:omega} hold and let Assumption \ref{assumption:phi} hold with $g\equiv0$. Then, {ITALEX-CT} with the approximation oracle presented in Algorithm~\ref{alg:approx_alg} using GCG or PG as the iteration of $\mathcal{A}$ and the expansion oracle defined in Proposition~\ref{prop:expansion_new},
outputs a vector {$\bx^R$, with $R\leq \left\lceil\log_2\left(\frac{\epsilon_1}{\varepsilon}\right)\right\rceil+1$, such that  
$$\varphi({\bx^R})-\varphi^*\leq\varepsilon,\;\text{ and } \omega({\bx^R})\leq \omega^*,$$}
after at most {$K_1+K_2+N$} inner (GCG/PG) iterations, where
{
\begin{align*}
     K_1&=\log_{2}\left(\min\left\{\frac{1}{ L_f \mathcal{D}_{\Lev_{\omega}(\omega^*)}^2},\frac{4}{\epsilon_1}\right\}\left(\varphi(\bx^0)-\bar{\varphi}_1\right)\right), \\
     K_2&=\frac{64L_f \mathcal{D}_{\Lev_{\omega}(\omega^*)}}{\varepsilon}+\left(\log_{2}\left(9\right)+2\right)\left(\left\lceil\log_2\left(\frac{\epsilon_1}{\varepsilon}\right)\right\rceil+1\right),
     \\ N&=\left\lceil \left(\frac{2L_f}{\varepsilon}\right)^{\sfrac{\kappa}{2}}\gamma(\omega^*-\omega(\bx^0))\right\rceil+\left\lceil\log_2\left(\frac{\epsilon_1}{\varepsilon}\right)\right\rceil+1.
\end{align*}}
\end{corollary}

\section {Numerical results}\label{section:Numerical}
In this section, we present a set of numerical experiments for comparing ITALEX-CT with existing methods. The general setup considers an inner function $\varphi=f+g$ where the smooth function $f$ is given by 
\begin{equation*}
    f(\bx)=\|\bA\bx-\tilde{\bb}\|^2,
\end{equation*} 
which has a $L_f$-Lipschitz continuous gradient with $L_f=\lambda_{\max}{(\bA^T\bA)}$.
Each experiment differs in the choice of  the nonsmooth $g$ and outer function $\omega$. {The functions $\omega$ satisfy Assumption~\ref{assumption:omega}, where $\gamma$ and $\kappa$ are given by Table~\ref{tbl:error_bound}.}

Following the numerical experiments in \cite{MNG}, 
we used  the ``regularization tools" MATLAB package \cite{R_ToolBox} 
to generate problems of the form $\bA\bx=\bb$. Thus, the parameter $\bA\in\Real^{n\times n}$ and the variable $\bx\in\Real^n$ with $n=1000$ were generated exactly by the functions "phillips","foxgood","baart" of the package (unless stated otherwise, see Sections~\ref{sec:numerical_nonsmooth_stronglyconvex} and \ref{sec:numerical_nonsmooth}). For each setting we generate $I=50$ instances of the problem where in each instance $i$, the vector $\tilde{\bb}=\bA\bx+\bnu^i$ and each component of $\bnu^i$ is sampled from a normal distribution with $\mu=0$ and $\sigma=10^{-2}$.

In each experiment, we compared two variations of the {ITALEX-CT method, referred to hereafter as ITALEX, }where the approximation oracle is implemented by either PG or GCG, with the methods mentioned in Table~\ref{table:Comparison}. We let each of the methods run for 30 minutes. {The total number of the PG/GCG iterations of method $m$ in realization $i$ is denoted as $\bar{M}_{i,m}$.
We denote the set of methods compared in an experiment by $\mathcal{M}$.}
 For each realization $i$ and tested method $m$ we
define $\{\bx^k_{i,m}\}_{k\in\mathbb{N}}$ as the sequence produced by the method, 
{and for each time $t$, we compute $\tilde{\bx}^t_{i,m}$ as the last iterate that was produced 
before time $t$.}
{Thus, we compare the average \emph{inner function} normalized optimality gap, given by $$\Delta\varphi_{t,m}=\frac{1}{I}\sum_{i=1}^I\frac{\varphi(\tilde{\bx}^t_{i,m})-\varphi_i^*}{\|\bx_i^*\|^2},$$ where $\varphi_i^*$ denotes the optimal value of the inner function, which was computed using CVX \cite{cvx} for MATLAB and Gurobi version 9.1.1 solver \cite{gurobi}, and $\bx^*_i$ denotes the approximated optimal solution of the bilevel problem. We note that the use of $\|\bx^*_i\|^2$ in the denominator is due to the fact that the guarantees for the algorithms are usually in the form of $\frac{\varphi(\bx^k)-\varphi^*}{\|\bx^0-\bx^*\|^2}$ (\cite{Big_SAM}, \cite{MNG}), and in this paper we provide a guarantee of the form $\frac{\varphi(\bx^k)-\varphi^*}{D_{\omega^*}^2}$ which also scales up with $\|\bx^*\|^2$. Since the optimal solution of the bilevel problem is unknown, $\bx^*_i$ was taken as the iterate achieving the minimal $\varphi$ value by any method tested.}

To compare the value of the \textit{outer} function, we could not use off-the shelf solvers to solve \eqref{prob:MNP'} in order to evaluate $\omega^*_i$, as we encountered the numerical problems described in Subsection~\ref{subsection:challenges}. Therefore, we compare the ratio between the \textit{outer} function values to the maximum value achieved by any of the algorithms, that is, denoting $\omega_{max}^i=\max_{ m\in \mathcal{M}, k}\left\{\omega\left(\bx^k_{i,m}\right)\right\}$ we measure $$\Delta\omega_{t,m}=\frac{1}{I}\sum_{r=1}^I 1-\frac{\omega(\tilde{x}^t_{i,m})}{\omega_{max}^i}.$$ Notice that a higher value of $\Delta\omega$ corresponds with a lower value of $\omega$. This measure will indicate how close each method is to $\omega^*_i$ for each realization $i$. We would generally expect that lower level of $\Delta\varphi$ will correspond with higher levels of $\omega$ (and therefore, lower levels of $\Delta \omega$).
Since \cite{MNG} demands that $\bx^0$ be $\argmin_{\bx\in\Real^n}\{\omega(\bx)\}$, and ITALEX demands that $\omega(\bx^0)\leq\omega^*$, we used $\bx^0=\bzero$ in all the experiments.
All experiments were ran on an Intel(R) Xeon(R) Gold 6254
CPU @ 3.10 GHz with a total of 300 GB RAM and 72 threads,
using MATLAB 2019a, where each method was allowed to run
on only one thread.

\subsection{Smooth and  strongly convex outer function}

In the first experiment, we used the same setup as in the numerical experiments in \cite{Big_SAM}, solving integral equations. Thus, $g(\bx)\equiv \delta_{\Real^n_{+}}$ is the indicator function of the non-negative orthant, $f$ is as described above, and the outer function is defined as $\omega(\bx)=\|\bx\|_{\bQ}=\sqrt{\bx^\top\bQ\bx}$,
where the positive definite matrix $\bQ=\mathbf{L}^\top\mathbf{L}+\mathbf{I}$, and $\mathbf{L}$, generated by the function $'get\_l(1000,1)'$ from \cite{R_ToolBox}, approximates the first-derivative operator. In this experiment, since optimizing over $\omega$ or $\omega^2$ is equivalent, we can compare ITALEX with methods that require both strong convexity and smoothness of the outer function: BiG-SAM \cite{Big_SAM}, iBiG-SAM \cite{iBiG}, IR-PG \cite{Solodov}, and MNG \cite{MNG}. 
The GCG implementation of ITALEX requires a bounded feasible set. Therefore, we utilized $\tilde{D}_\alpha(\by)$ from Lemma~\ref{lemma:D_alpha}, and for any fixed level set $\alpha_k$ and inner iteration $\by^j\in\Real^{2n}$ we redefined $\hat{g}^{\alpha_k}=\delta_{C_{\alpha_k}(\by^j)}+\delta_{\Lev_{\omega}(\alpha_k)}$, where $C_\alpha(\by)=\{\bx\in\Real^n_+:\norm{\bx}\leq \tilde{D}_\alpha(\by)\}$, and applied the GCG step.

\begin{figure}[b!]
\centering
\includegraphics[width=0.9\textwidth]{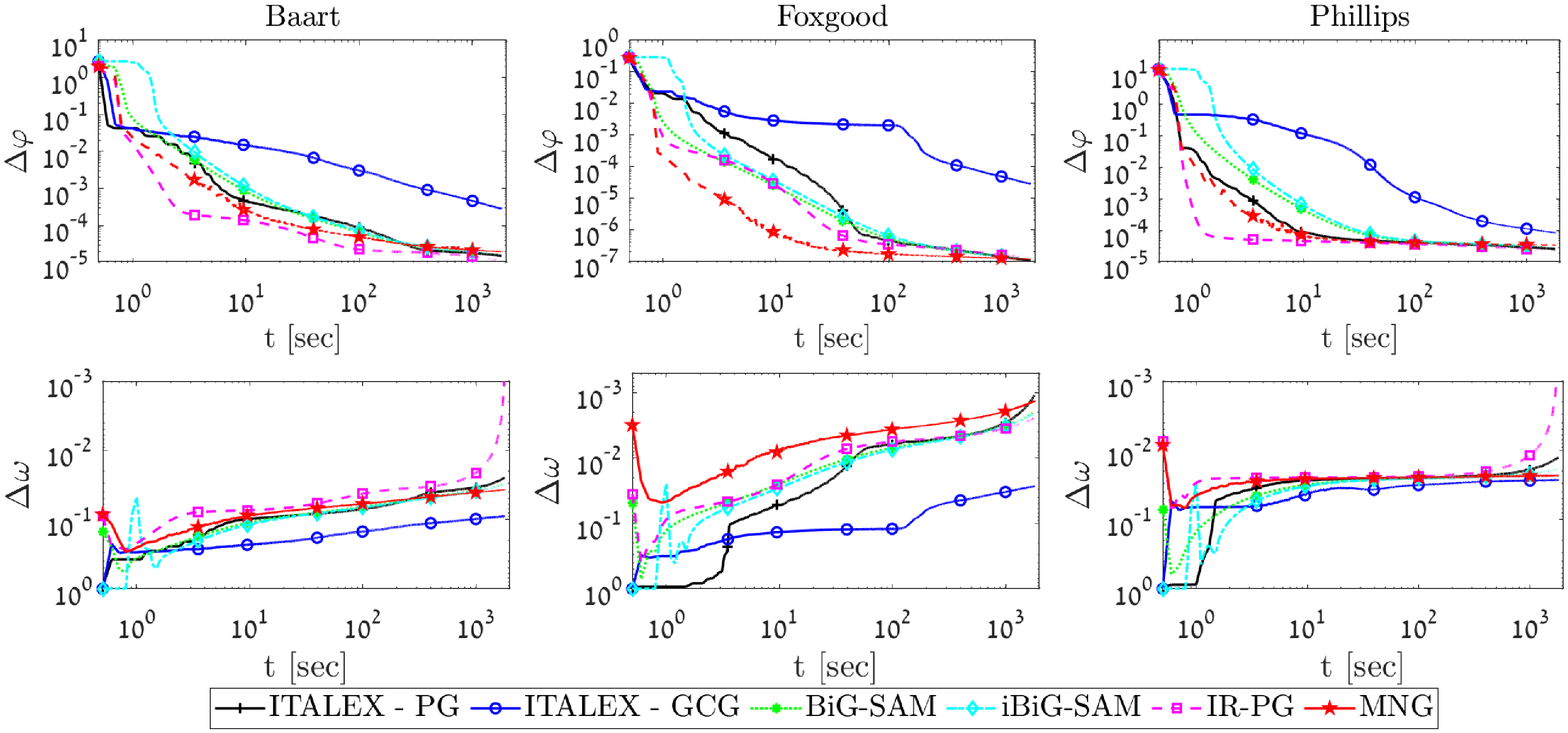}
\captionsetup{justification=centering}
\caption{$\Delta\varphi$ and $\Delta\omega$ over time for $\varphi(\bx)=\|\mathbf{A}\bx-\mathbf{b}\|^2+\delta_{\Real^n_{+}}$ and $\omega(\bx)=\|\bx\|_{\bQ}$.}\label{fig:exp1}
\end{figure}

\begin{table}[t]
\centering
\begin{minipage}{\textwidth}
\centering
\small
        \begin{tabular}{ |p{1.8cm}||p{1.4cm}|p{1.4cm}|p{1.4cm}|p{1.4cm}|p{1.4cm}|p{1.4cm}|  }\hline \multirow{3}{2.4cm}{Method}&
        \multicolumn{2}{c|}{Baart}&
        \multicolumn{2}{c|}{Foxgood}&
        \multicolumn{2}{c|}{Phillips}\\
         &   $\Delta\varphi$&   $\Delta\omega$&   $\Delta\varphi$&   $\Delta\omega$&   $\Delta\varphi$&   $\Delta\omega$\\\hline 
        ITALEX - PG&$1.41$e$-5$ $(1.7$e$-5)$&$1.87$e$-2$ $(2.7$e$-2)$&$\mathbf{1.10}$\textbf{e}$\mathbf{-7}$ $(8.3$e$-8)$&$8.49$e$-4$ $(9.5$e$-4)$&$2.74$e$-5$ $(6.2$e$-6)$&$1.52$e$-2$ $(1.2$e$-2)$\\ \hline ITALEX - GCG&$2.83$e$-4$ $ (6.9$e$-5)$&$8.66$e$-2$ $ (5.0$e$-2)$&$3.49$e$-5$ $ (4.1$e$-6)$&$2.87$e$-2$ $ (3.4$e$-3)$&$8.77$e$-5$ $ (1.5$e$-5)$&$3.15$e$-2$ $ (1.7$e$-2)$\\ \hline BiG-SAM&$1.54$e$-5$ $ (2.0$e$-5)$&$2.44$e$-2$ $ (3.5$e$-2)$&$1.12$e$-7$ $ (8.1$e$-8)$&$1.13$e$-3$ $ (9.8$e$-4)$&$3.11$e$-5$ $ (7.1$e$-6)$&$2.22$e$-2$ $ (1.3$e$-2)$\\ \hline iBiG-SAM&$1.56$e$-5$ $ (2.0$e$-5)$&$2.52$e$-2$ $ (3.5$e$-2)$&$1.14$e$-7$ $ (8.2$e$-8)$&$1.30$e$-3$ $ (9.6$e$-4)$&$3.12$e$-5$ $ (7.3$e$-6)$&$2.23$e$-2$ $ (1.3$e$-2)$\\ \hline IR-PG&$\mathbf{1.10}$\textbf{e}$\mathbf{-5}$ $ (1.3$e$-5)$&$1.27$e$-4$ $ (3.9$e$-4)$&$1.26$e$-7$ $ (9.4$e$-8)$&$1.82$e$-3$ $ (1.1$e$-3)$&$\mathbf{2.34}$\textbf{e}$\mathbf{-5}$ $ (5.2$e$-6)$&$0.00$e$+0$ $ (0.0$e$+0)$\\ \hline MNG&$1.80$e$-5$ $ (2.3$e$-5)$&$2.96$e$-2$ $ (4.3$e$-2)$&$1.17$e$-7$ $ (9.9$e$-8)$&$9.06$e$-4$ $ (1.6$e$-3)$&$3.75$e$-5$ $ (9.9$e$-6)$&$2.78$e$-2$ $ (1.7$e$-2)$\\ 
      [1ex] 
 \hline
\end{tabular}
\end{minipage}
\caption{Mean error (standard deviation) after 30 minutes for ${\varphi(\bx)=\|\mathbf{A}\bx-\mathbf{b}\|^2+\delta_{\Real^n_{+}}}$ and $\omega(\bx)=\|\bx\|_{\bQ}$.}
\label{table:exp1}
\vskip-10pt
\end{table}

The values of $\Delta\varphi_{t,m}$ and $\Delta\omega_{t,m}$ over time are presented in Figure~\ref{fig:exp1}. Table~\ref{table:exp1} summarizes these values for $t$ equal to 30 minutes.
As can be seen,  while there is not a consistently superior method in all three data sets, ITALEX-GCG is consistently the slowest. However, for the Foxgood data set, ITALEX-PG achieves the best performance, and for the other two  data sets, ITALEX-PG achieves the second best performance, second only IR-PG. The slower rate of ITALEX-GCG can be explained by the large diameter of the constructed compact domain needed to apply the GCG step.
All methods exhibit a clear trade-off between the \textit{inner} and \textit{outer} function values.

\subsection{Nonsmooth and strongly convex outer function}\label{sec:numerical_nonsmooth_stronglyconvex}

In the next experiment, we want to evaluate the performance of the different methods in a sparse setting. We therefore chose $\tilde{\bb}=\bA\tilde{\bx}+\bnu^i$, where $\tilde{\bx}$ is a sparse vector, such that
$\tilde{\bx}\in\Real^{1,000},\|\tilde{\bx}\|_0=200$ (taking 200 randomly selected components of $\bx$), and $\bnu^i$ is generated by normal distribution with mean $\bzero$ and standard deviation $\sigma=10^{-3}$.
In this set of experiments, we take $g\equiv0$ and choose $\omega(\bx)$ to be the strongly convex and nonsmooth elastic-net function $\omega(\bx)=\|\bx\|_1+\rho\|\bx\|^2$ with $\rho=0.05$. We compared both implementations of ITALEX (with GCG and PG) to the only two methods that can deal with this type of outer function: IR-IG \cite{Incremental} and BiG-SAM \cite{Big_SAM} applied to a smoothed function $\omega$, with two different uniform-accuracy smoothing parameters $\delta=1$ and ${\delta=1e-3}$. The smoothing is done, as described in \cite{Big_SAM}, by taking $\bx-s\nabla\tilde{\omega}(\bx)$ instead of $\bx-s\nabla \omega(\bx)$, with $\tilde{\omega}=H_{\delta}+\rho\|\bx\|^2$,  where $H_{\delta}$ is the Huber function, and $s=2\delta$, where  (as the $l_1$ norm is 1-Lipschitz continuous). The results are shown in Figure~\ref{fig:exp3} and Table~\ref{table:exp3}.

\begin{figure}[t]
\centering
\includegraphics[width=1\textwidth]{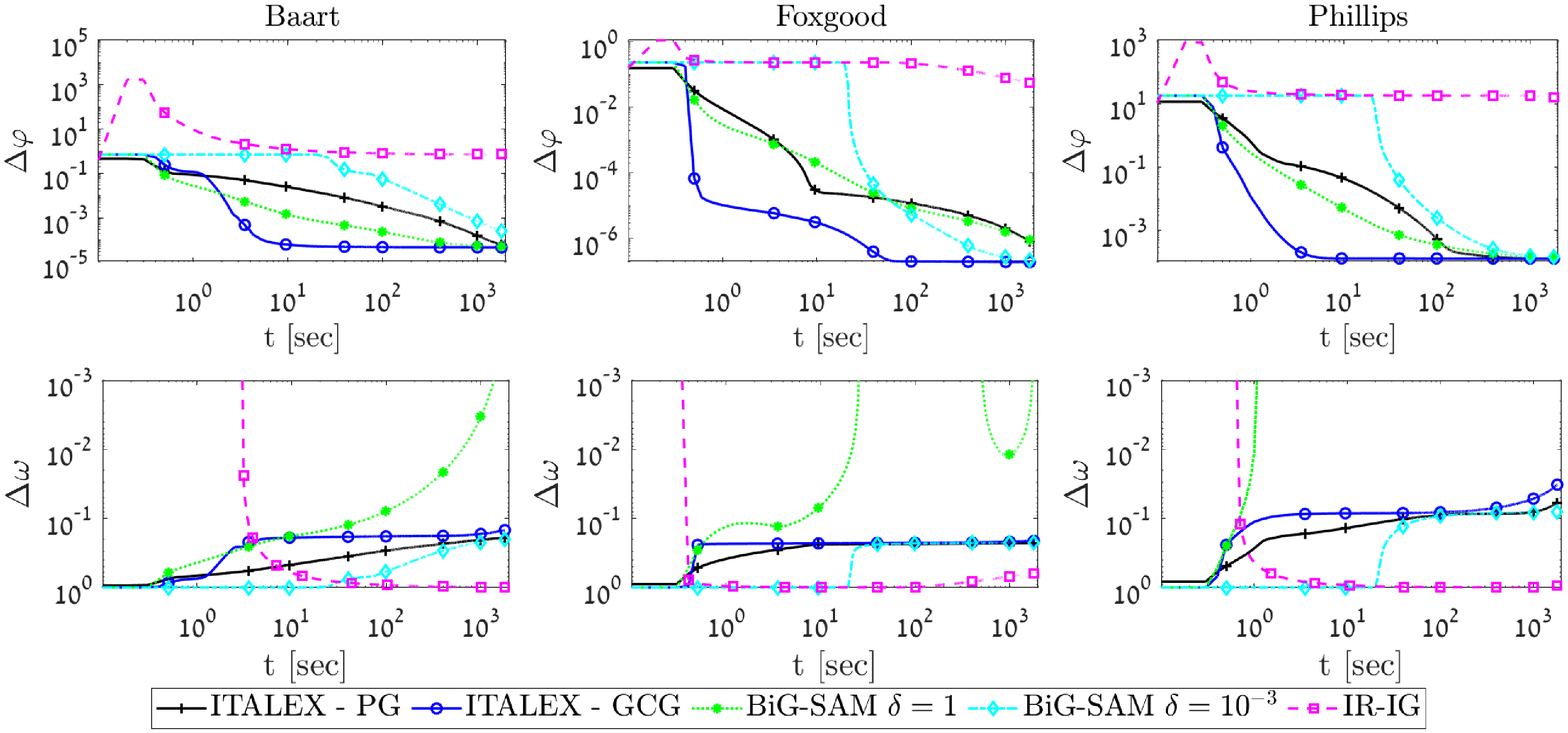}

\caption{$\Delta\varphi$ and $\Delta\omega$ over time for $\varphi(\bx)=\|\mathbf{A}\bx-\mathbf{b}\|^2$, $\omega(\bx)=\|\bx\|_1+\rho\|\bx\|^2$.}\label{fig:exp3}
\end{figure}

Since in the first steps of IR-IG the value of $\omega$ increases significantly, we took $\omega^*_i=\max_{k\geq 500, m\in M}\left\{\omega\left(\bx^k_{i,m}\right)\right\}.$ Therefore, for larger values of $\omega$ (that appear only in IR-IG) we cannot show $\Delta\omega$.
For all the data sets, ITALEX with  GCG is the fastest and IR-IG is the slowest, in terms of inner function convergence, where the latter achieves $\Delta\varphi$ values that are larger by five orders of magnitude from the former. Moreover, while ITALEX with PG and BiG-SAM with smoothing parameter $\delta=1$ have similar convergence with respect to the inner problem, BiG-SAM obtains significantly higher (suboptimal)  \textit{outer} function values {due to the rough smoothing.} 
Moreover, BiG-SAM with $\delta=10^{-3}$ converges noticeably slower compared to ITALEX, especially in the beginning, probably due to the complexity of smoothed BiG-SAM, which is $O(1/\varepsilon\delta^2)$.

\begin{table}[H]
	\centering
	\begin{minipage}{\textwidth}
		\centering
		\small
		\begin{tabular}{ |p{1.8cm}||p{1.4cm}|p{1.4cm}|p{1.4cm}|p{1.4cm}|p{1.4cm}|p{1.4cm}|  }\hline \multirow{3}{2.4cm}{Method}&
			\multicolumn{2}{c|}{Baart}&
			\multicolumn{2}{c|}{Foxgood}&
			\multicolumn{2}{c|}{Phillips}\\
			&   $\Delta\varphi$&   $\Delta\omega$&   $\Delta\varphi$&   $\Delta\omega$&   $\Delta\varphi$&   $\Delta\omega$\\\hline 
			ITALEX - PG&$1.78$e$-5$ $(1.1$e$-5)$&$1.91$e$-1$ $(8.8$e$-3)$&$6.80$e$-7$ $(5.5$e$-8)$&$2.50$e$-1$ $(9.2$e$-4)$&$6.45$e$-5$ $(1.9$e$-5)$&$1.76$e$-1$ $(1.6$e$-2)$\\ \hline ITALEX - GCG&$\mathbf{9.05}$\textbf{e}$\mathbf{-7}$ $ (8.0$e$-6)$&$1.47$e$-1$ $ (3.5$e$-2)$&$\mathbf{5.50}$\textbf{e}$\mathbf{-8}$ $ (3.0$e$-8)$&$2.35$e$-1$ $ (7.1$e$-3)$&$\mathbf{6.39}$\textbf{e}$\mathbf{-5}$ $ (1.9$e$-5)$&$1.53$e$-1$ $ (1.2$e$-2)$\\ \hline BiG-SAM $\delta=1$&$5.31$e$-6$ $ (7.4$e$-6)$&$2.99$e$-4$ $ (6.5$e$-4)$&$7.69$e$-7$ $ (8.4$e$-8)$&$2.93$e$-2$ $ (5.7$e$-3)$&$7.39$e$-5$ $ (1.9$e$-5)$&$1.24$e$-1$ $ (6.7$e$-3)$\\ \hline BiG-SAM $\delta=10^{-3}$&$2.12$e$-4$ $ (1.7$e$-5)$&$2.00$e$-1$ $ (9.0$e$-3)$&$8.82$e$-8$ $ (3.3$e$-8)$&$2.47$e$-1$ $ (9.3$e$-4)$&$7.67$e$-5$ $ (1.9$e$-5)$&$1.95$e$-1$ $ (4.8$e$-4)$\\ \hline 
			IR-IG&$7.15$e$-1$ $ (5.1$e$-2)$&${9.90}${e}${-1}$ $ (2.8$e$-4)$&$5.68$e$-2$ $ (1.2$e$-3)$&${6.32}${e}${-1}$ $ (4.1$e$-3)$&$1.59$e$+1$ $ (2.0$e$-1)$&${9.51}$\textbf{e}${-1}$ $ (4.3$e$-3)$\\  [1ex] \hline
		\end{tabular}
	\end{minipage}
	\caption{Mean error (standard deviation) after 30 minutes for $\varphi(\bx)=\|\mathbf{A}\bx-\mathbf{b}\|^2)$ and $\omega(\bx)=\|\bx\|_1+\rho\|\bx\|^2$.}
	\label{table:exp3}
	\vskip-10pt
\end{table}

\subsection{Nonsmooth and 
non-strongly convex outer function}\label{sec:numerical_nonsmooth}

For our last set of experiments, we tackle the setting that motivated the development of ITALEX, the case where $\omega $ is neither smooth nor strongly convex. In these experiments we again choose $g\equiv 0$, and set $\omega(\bx)=\|\bx\|_1$. The data was generated as in the elastic-net experiment. Since no existing first-order methods can address this setting, we only show a comparison between the two ITALEX implementations in Figure~\ref{fig:exp4}.

We observe that in this example, in contrast to the previous one, ITALEX-CT with PG has a significantly faster inner function convergence rate than that ITALEX-CT with GCG. We believe this is due to the difference between the linear convergence rate of the PG compared to the sublinear convergence rate of GCG for optimizing the least square function over a polyhedral set \cite{Lin_Proj_Grad_Polyhedral}.
\begin{figure}[t]
\vskip-10pt
\centering
\includegraphics[width=1\textwidth]{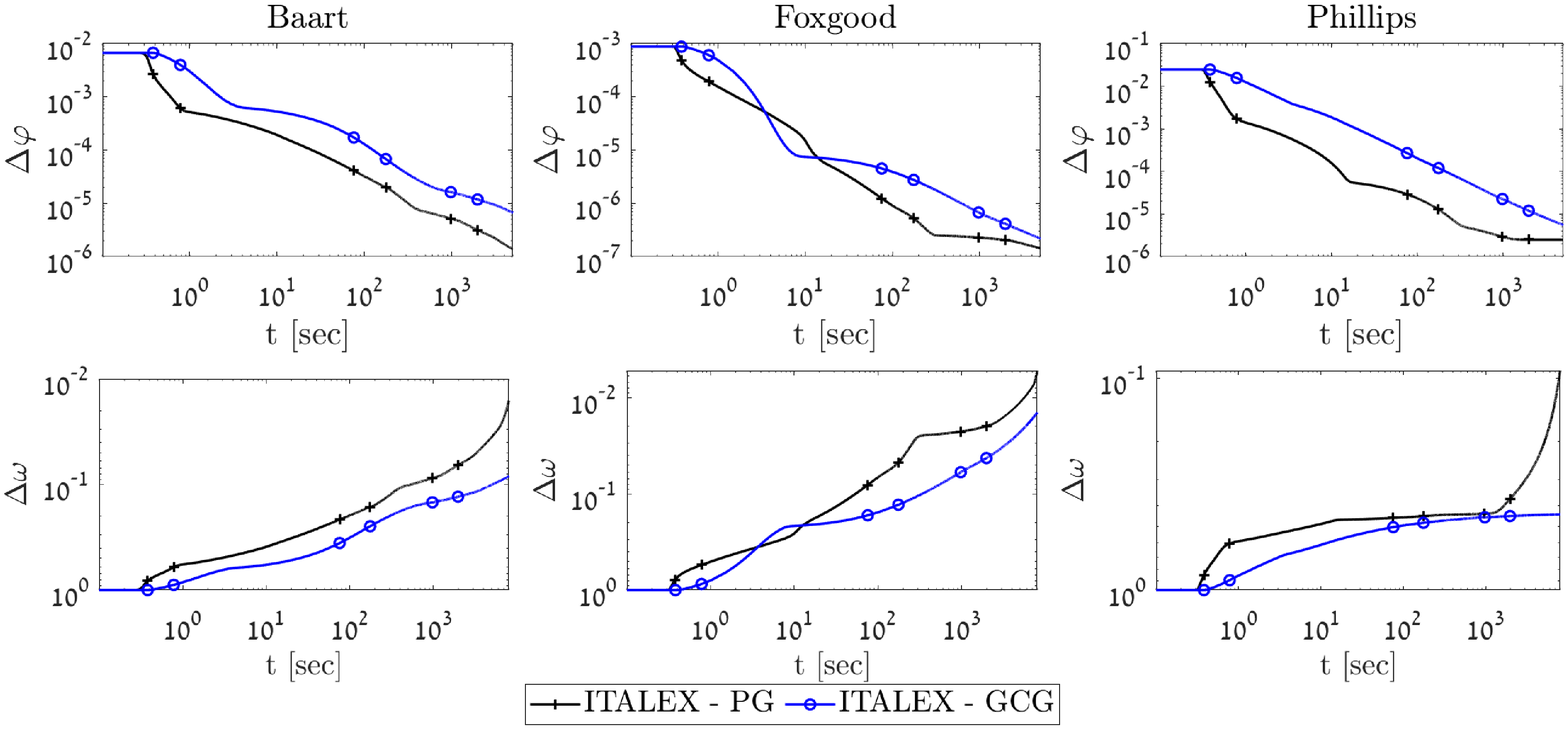}
\captionsetup{justification=centering}
\caption{$\Delta\varphi$ and $\Delta\omega$ over time for $\varphi(\bx)=\|\mathbf{A}\bx-\mathbf{b}\|^2$ and $\omega(\bx)=\|\bx\|_1$.}\label{fig:exp4}
\end{figure}
In order to compare ITALEX-CT against some benchmark, we compare its results to the ones obtained by iterative regularization.
That is, we solve the following problems
\begin{equation*}\label{prob:p_lambda}\tag{$P_\lambda$}
    \min \varphi(\bx)+\lambda_{\ell}\omega(\bx),
\end{equation*}
with decreasing values of $\lambda_{\ell}$. Specifically, we chose
$\lambda_{\ell}=\frac{1}{2^\ell}\lambda_{\max}(\bA^{\top}\bA)$, where $\ell\in[15]$. We solved each of the regularized problems using CVX with Gurobi solver, and computed the values of $\omega$ and $\varphi$. In Figure \ref{fig:exp4_om_phi} we present the trade-off graph between $\Delta\varphi(\bx)$ and $\Delta\omega(\bx)$ for both ITALEX-CT implementations and for the iterative regularization. The graph for ITALEX-CT was constructed by computing $\omega^*_i=\omega(\bx^i_{\lambda_{15}})$ for each instance $i$. For the ITALEX-CT method, for each instance $i$, variant $m$, and value of $\Delta\varphi$ we computed $$\Delta\omega^i_m(\Delta\varphi)=1-\frac{1}{\omega^*_i}\min_{k\in [K^{i,m}]}\{\omega(\bx^k_{i,m}):\Delta\varphi(\bx^k_{i,m})\leq\Delta\varphi\}.$$ We then computed the average and percentile value over the realizations. For the  iterative regularization graph, for each value $\lambda_\ell$, given $\bx_{\lambda_{\ell}}^i$, the solution \eqref{prob:p_lambda} for  the $r$th realization, we computed $\Delta\varphi(\bx_{\lambda_{\ell}}^r)$  and $\Delta\omega(\bx_{\lambda_{\ell}}^r)$.
We then computed the average and percentiles of these values over the realizations.

\begin{figure}[ht]
\centering
\includegraphics[width=1\textwidth]{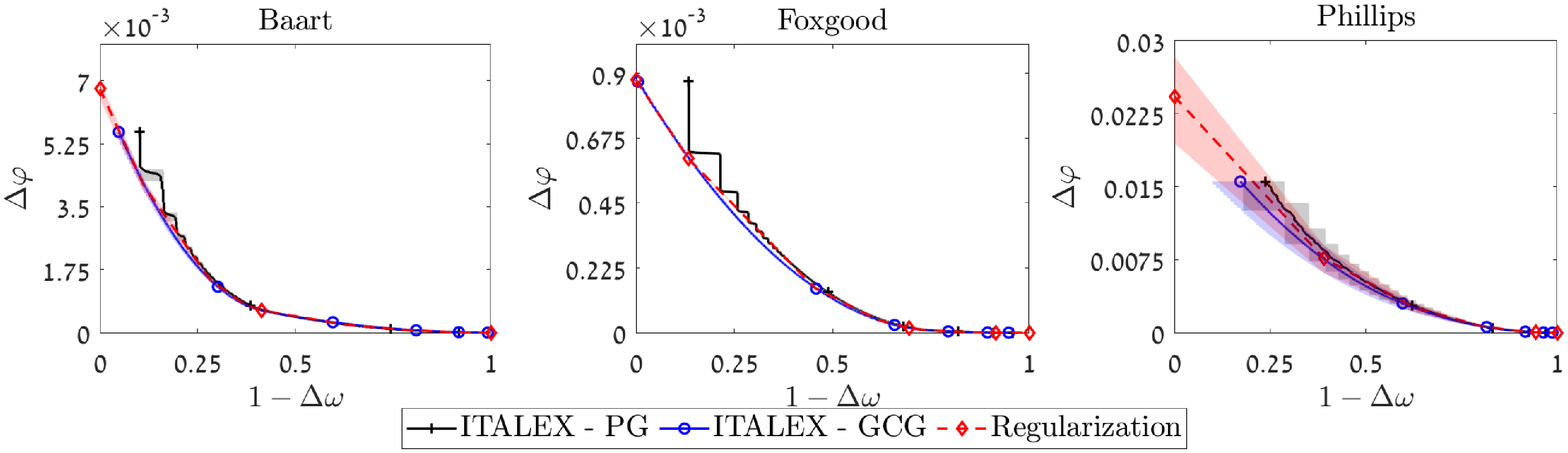}
\caption{$\Delta\varphi$ vs. $1-\Delta\omega$ for  $\varphi(\bx)=\|\mathbf{A}\bx-\mathbf{b}\|^2$ and $\omega(\bx)=\|\bx\|_1$. The lines indicate the average values, while the shaded regions indicate the area between the 10th and 90th percentiles}\label{fig:exp4_om_phi}
\end{figure}

In Figure~\ref{fig:exp4_om_phi}, we see that during its execution, ITALEX follows the line of the iterative regularization, and is sometimes better (below it), {this pattern stems from the fact that ITALEX follows $H(\bx,\alpha)$ for increasing $\alpha$ and the iterative regularization essentially solves the Lagrangian relaxation of $H(\bx,\alpha)$ (see definition in Section~4) with $\lambda$ being an optimal value of the dual variable corresponding to some (unknown) $\alpha$. Nevertheless, contrary to iterative regularization, ITALEX does not require solving multiple optimization problems to obtain the above lines.} Interestingly ITALEX with PG exhibits a 'step' structure {that occurs due to PG making large expansion steps and, at the same time, a large number of iterations in each approximation oracle. This `step' structure also occurs in ITALEX with GCG, however, due to the smaller expansion steps which stem from the difference in the optimality measures of PG and GCG, this structure disappears when averaging across realizations. }

\section{Discussion and conclusions}
In this paper, we presented a general methodology called ITALEX, to solve simple convex bilevel optimization problems. We showed that our methodology guarantees convergence to an optimal solution of the bilevel problem, and that it can be applied under milder assumptions about the structure of the outer function than the assumptions used in previous work. Specifically, ITALEX-CT does not require the outer function to be strongly convex or smooth. We also presented two implementations of ITALEX-CT  using GCG and PG, which obtain an $O(1/\varepsilon)$ iteration complexity for the inner function convergence, the same as the best known complexity for the case of smooth and strongly convex outer function, and is the first method to show an $O(1/\sqrt{\varepsilon})$ iteration complexity for the outer function convergence.

As to future work, we believe that under more restrictive assumptions on the structure of $\varphi$ and $\omega$, an accelerated convergence of ITALEX-CT with respect to the inner function may be obtained. It is also of interest to identify the classes of problems for which PG works better than GCG and vice versa. 
\subsection*{Acknowledgments}
We are grateful to Prof. Shoham Sabach for presenting us with this problem and for his helpful
suggestions that greatly improved this work.

\begin{appendices}
\section{Proof of Lemma~\ref{lemma:h_decreasing}}\label{appendix:h_properties}
\begin{proof}
\begin{enumerate}[label=(\roman*)]
    \item 
    {$H(\bx,\bz,\alpha)$ is an extended value function that is jointly convex in $\bx$, $\bz$, and $\alpha$, as a sum of the convex functions $\varphi(\bx)+\norm{\bx-\bz}^2$ and the indicator function over the convex set $\epi(\omega)$.} Since by definition $h(\alpha)$ is a partial minimization of convex function {$H(\bx,\bz,\alpha)$}, by \cite[Theorem 2.18]{FO_Book}, it is also convex.
    \item First, we will show $h$ is a nonincreasing function. Let $\alpha_1\leq\alpha_2$, since $\Lev_{\omega}(\alpha_1)\subseteq\Lev_{\omega}(\alpha_2)$, it is clear that
    $${h(\alpha_2)=\hskip-7pt \min_{\substack{\bx\in\Real^n,\\ \bz\in\Lev_{\omega}(\alpha_2)}}\{\varphi(\bx)+\norm{\bx-\bz}^2\}\leq\hspace{-10pt}\min_{\substack{\bx\in\Real^n,\\ \bz\in\Lev_{\omega}(\alpha_1)}}
    \{\varphi(\bx)+\norm{\bx-\bz}^2\}=h(\alpha_1).}$$
    By the definition of $\omega^*$, for any $\alpha\geq\omega^*$ we have that $h(\alpha)=\varphi^*=h(\omega^*)$, and for any $\alpha<\omega^*$, we have that $h(\alpha)>\varphi^*$. Therefore, it remains to show that for any $\ubar{\omega}\leq\alpha_1<\alpha_2<\omega^*$, $h(\alpha_1)>h(\alpha_2)$.
    The choice of $\alpha_2$ implies that there exists $\lambda\in(0,1)$ such that $\alpha_2=\lambda \alpha_1+(1-\lambda)\omega^*$. Thus, by the convexity of $h(\cdot)$ we have that 
    $$h(\alpha_2)\leq\lambda h(\alpha_1)+(1-\lambda)h(\omega^*)<h(\alpha_1),$$
    where the final inequality follows from the fact that $\alpha_1<\omega^*$, and therefore $h(\alpha_1)>\varphi^*$, thus concluding the proof.
\end{enumerate}\vskip-10pt
\end{proof}

\section{Proof of Lemma~\ref{lemma:decreasing_Sequence} }\label{appendix:proof_lemma_decreasing_sequence}
\begin{proof} First denote $a=\max\{\frac{2}{\eta},\xi_1\}$. We will prove the statement by induction. In the case $p=1$, $\xi_1\leq \frac{a}{1}={\max\{\frac{2}{\eta},\xi_1\}}$ trivially holds. Now, assume that the statement is true for $j=1,...,p$, we will prove it also holds for $p+1$. By the properties of the sequence and the induction assumption, we have \begin{equation}\label{eq:technical_lemma_1}
    (1+\eta\xi_{p+1})\xi_{p+1}\leq\xi_p\leq\frac{a}{p}
\end{equation}
Assume to the contrary that $\xi_{p+1}>\frac{a}{p+1}$, then by  {multiplying \eqref{eq:technical_lemma_1} by $p(p+2)^2$} we obtain 
\begin{equation*}
    {p^2+p+\eta a p}< {p^2+2p+1}.
\end{equation*}
By definition of $a$, the above inequality implies
$$2p=  \frac{2}{\eta} \eta p\leq \eta a p < p+1,$$ 
and since $p\geq 1$, we obtain a contradiction. Thus, $\xi_{p+1}\leq \frac{a}{p+1}$. 
\end{proof}
\end{appendices}

\bibliographystyle{spmpsci}      
\bibliography{1.bib}   

\end{document}